\documentclass[12pt]{amsart}
\newcommand{\deleted}[1]{}
\newcommand{\delete}[1]{}
\newcommand{\mynotes}[1]{}
\newcommand\notes[1]{}
\usepackage{txfonts}
\usepackage{amssymb}
\usepackage{eucal}
\usepackage{graphicx}
\usepackage{amsmath}
\usepackage{amscd}
\usepackage[all]{xy}           

\usepackage{amsfonts,latexsym}
\usepackage{xspace}
\usepackage{epsfig}
\usepackage{float}
\usepackage{color}
\usepackage{fancybox}
\usepackage{colordvi}
\usepackage{multicol}
\usepackage{tikz}

\usepackage{wasysym}
\usepackage{xspace}
\usepackage[colorlinks,final,backref=page,hyperindex,hypertex]{hyperref}
\usepackage[active]{srcltx} 
\usepackage{mathtools} 

\newcommand\changed[1]{#1}
\changed{}

\newtheorem{theorem}{Theorem}[section]
\newtheorem{lemma}[theorem]{Lemma}

\newtheorem{prop}[theorem]{Proposition}
\theoremstyle{definition}
\newtheorem{defn}[theorem]{Definition}
\newtheorem{remark}[theorem]{Remark}
\newtheorem{exam}[theorem]{Example}

\deleted{\newtheorem{theorem}{Theorem}[section]
\newtheorem{prop-def}{Proposition-Definition}[section]
\newtheorem{coro-def}{Corollary-Definition}[section]

}

\newcommand{\nc}{\newcommand}
\nc{\tred}[1]{\textcolor{red}{#1}} \nc{\tblue}[1]{\textcolor{blue}{#1}} \nc{\tgreen}[1]{\textcolor{green}{#1}} \nc{\tpurple}[1]{\textcolor{purple}{#1}} \nc{\btred}[1]{\textcolor{red}{\bf #1}} \nc{\btblue}[1]{\textcolor{blue}{\bf #1}} \nc{\btgreen}[1]{\textcolor{green}{\bf #1}} \nc{\btpurple}[1]{\textcolor{purple}{\bf #1}}

\renewcommand{\Bbb}{\mathbb}

\newcommand{\efootnote}[1]{}

\newcommand\wyscco[1]{}
\renewcommand{\textbf}[1]{}

\nc{\mlabel}[1]{\label{#1}}  
\nc{\mcite}[1]{\cite{#1}}  
\nc{\mref}[1]{\ref{#1}}  
\nc{\mbibitem}[1]{\bibitem{#1}} 

\delete{
\nc{\mlabel}[1]{\label{#1}{\hfill \hspace{1cm}{\bf{{\ }\hfill(#1)}}}}
\nc{\mcite}[1]{\cite{#1}{{\bf{{\ }(#1)}}}}  
\nc{\mref}[1]{\ref{#1}{{\bf{{\ }(#1)}}}}  
\nc{\mbibitem}[1]{\bibitem[\bf #1]{#1}} 
}

\renewcommand\geq{\geqslant}
\renewcommand\leq{\leqslant}

\renewcommand\bar[1]{\overline{#1}}
\renewcommand\tilde[1]{\widetilde{#1}}

\nc\kdot{\bfk}
\nc\simple{simple\xspace}

\nc{\rbw}{\mathfrak{R}} \nc{\brp}{\mathrm{brp}} \nc{\lead}{\mathrm{Lead}} \nc{\Id}{\mathrm{Id}} \nc{\Irr}{\mathrm{Irr}} \nc{\vx}{\sigma} \nc{\vy}{\tau} \nc{\dvx}{\sigma^{(1)}} \nc{\dvy}{\tau^{(1)}} \nc{\done}{\vep} \nc{\citep}[1]{\cite{#1}} \nc{\wt}{\mathrm{wt}} \nc{\bre}[1]{|#1|} \nc{\mapmonoid}{\frakM} \nc{\disjoint}{\frakM'}
\nc{\ncpoly}[1]{\langle #1\rangle}  
\nc{\mapm}[1]{\frakM(#1)}
\nc{\diff}[1]{{}^\NC\{ #1 \}} \nc{\disj}[1]{\{{#1}\}'} \nc{\mdisj}[1]{\frakM'(#1)} \nc{\brho}{\bar{\rho}} \nc{\om}{\bar{\frakm}} \nc{\frakn}{\mathfrak n} \nc{\ddeg}[1]{^{(#1)}} \nc{\opset}{X} \nc{\genset}{{Z}} \nc{\NC}{\mathrm{{NC}}} \nc{\leaf}{\mathrm{leaf}} \nc{\twig}{\mathrm{twig}} \nc{\fe}{\mathrm{fl}} \nc{\munderline}[1]{#1} \nc{\bo}{o} \nc{\dep}{\mathrm{dep}} \nc{\ofe}{\mathrm{ofl}} \nc{\dfe}{\mathrm{dfe}} \nc{\fex}{\mathrm{fex}} \nc{\dl}{\mathrm{dlex}} \nc{\db}{\mathrm{db}} \nc{\lex}{\mathrm{lex}} \nc{\clex}{\mathrm{clex}} \nc{\dgp}{\mathrm{dgp}} \nc{\dgx}{\mathrm{dgx}} \nc{\br}{\mathrm{br}} \nc{\obd}{\mathrm{odb}} \nc{\ob}{\mathrm{ob}}
\nc{\loc}{location\xspace}
\nc{\occ}{occurrence\xspace}
\nc{\occs}{occurrences\xspace}
\nc{\pla}{placement\xspace}
\nc{\plas}{placements\xspace}

\nc{\bin}[2]{ (_{\stackrel{\scs{#1}}{\scs{#2}}})}  
\nc{\binc}[2]{ \left (\!\! \begin{array}{c} \scs{#1}\\
    \scs{#2} \end{array}\!\! \right )}  
\nc{\bincc}[2]{  \left ( {\scs{#1} \atop
    \vspace{-1cm}\scs{#2}} \right )}  
\nc{\bs}{\bar{S}} \nc{\cosum}{\sqsubset} \nc{\la}{\longrightarrow} \nc{\rar}{\rightarrow} \nc{\dar}{\downarrow} \nc{\dprod}{**} \nc{\dap}[1]{\downarrow \rlap{$\scriptstyle{#1}$}} \nc{\md}{\mathrm{dth}} \nc{\uap}[1]{\uparrow \rlap{$\scriptstyle{#1}$}} \nc{\defeq}{\stackrel{\rm def}{=}} \nc{\disp}[1]{\displaystyle{#1}} \nc{\dotcup}{\ \displaystyle{\bigcup^\bullet}\ } \nc{\gzeta}{\bar{\zeta}} \nc{\hcm}{\ \hat{,}\ } \nc{\hts}{\hat{\otimes}} \nc{\barot}{{\otimes}} \nc{\free}[1]{\widetilde{#1}} \nc{\uni}[1]{\tilde{#1}} \nc{\hcirc}{\hat{\circ}} \nc{\leng}{\ell} \nc{\lleft}{[} \nc{\lright}{]} \nc{\lc}{\lfloor} \nc{\rc}{\rfloor}
\nc{\lb}{[} 
\nc{\rb}{]} 
\nc{\curlyl}{\left \{ \begin{array}{c} {} \\ {} \end{array}
    \right.  \!\!\!\!\!\!\!}
\nc{\curlyr}{ \!\!\!\!\!\!\!
    \left. \begin{array}{c} {} \\ {} \end{array}
    \right \} }
\nc{\longmid}{\left | \begin{array}{c} {} \\ {} \end{array}
    \right. \!\!\!\!\!\!\!}
\nc{\onetree}{\bullet} \nc{\ora}[1]{\stackrel{#1}{\rar}}
\nc{\ola}[1]{\stackrel{#1}{\la}}
\nc{\ot}{\otimes} \nc{\mot}{{{\boxtimes\,}}} \nc{\otm}{\overline{\boxtimes}} \nc{\sprod}{\bullet} \nc{\scs}[1]{\scriptstyle{#1}} \nc{\mrm}[1]{{\rm #1}} \nc{\msum}{\sum\limits}
\nc{\margin}[1]{\marginpar{\rm #1}}   
\nc{\dirlim}{\displaystyle{\lim_{\longrightarrow}}\,} \nc{\invlim}{\displaystyle{\lim_{\longleftarrow}}\,} \nc{\mvp}{\vspace{0.3cm}} \nc{\tk}{^{(k)}} \nc{\tp}{^\prime} \nc{\ttp}{^{\prime\prime}} \nc{\svp}{\vspace{2cm}} \nc{\vp}{\vspace{8cm}} \nc{\proofbegin}{\noindent{\bf Proof: }}
\nc{\proofend}{$\blacksquare$ \vspace{0.3cm}}
\nc{\modg}[1]{\!<\!\!{#1}\!\!>}
\nc{\intg}[1]{F_C(#1)} \nc{\lmodg}{\!<\!\!} \nc{\rmodg}{\!\!>\!} \nc{\cpi}{\widehat{\Pi}}
\nc{\sha}{{\mbox{\cyr X}}}  
\nc{\shap}{{\mbox{\cyrs X}}} 
\nc{\shpr}{\diamond}    
\nc{\shp}{\ast} \nc{\shplus}{\shpr^+}
\nc{\shprc}{\shpr_c}    
\nc{\msh}{\ast} \nc{\zprod}{m_0} \nc{\oprod}{m_1} \nc{\vep}{\varepsilon} \nc{\labs}{\mid\!} \nc{\rabs}{\!\mid}
\nc{\astarrow}{\overset{\raisebox{-2pt}{{\scriptsize $\ast$}}}{\rightarrow}}
\nc{\astlarrow}{\overset{\raisebox{-2pt}{{\scriptsize $\ast$}}}{\longrightarrow}}
\nc{\lastarrow}{\overset{\raisebox{-2pt}{{\scriptsize $\ast$}}}{\leftarrow}}
\nc{\mastarrow}[1]{\overset{\raisebox{-2pt}{{\scriptsize $#1$}}}{\rightarrow}}
\nc{\quvarrow}[3]{#1 \overset{q,u,v}{\longrightarrow}_{#3} #2}
\nc{\quvkto}[1]{f_{#1} \overset{q_{#1}, u_{#1}, v_{#1}}{\longrightarrow}_\phi g_{#1}}
\nc{\tvarrow}[3]{#1 \overset{(t,v)}{\longrightarrow}_{#3} #2}
\nc{\Supp}{{\rm Supp}}
\nc{\mpu}{u^{\ast}}
\nc{\mpv}{v^{\ast}}
\nc{\mpw}{w^{\ast}}
\nc{\mpx}{x^{\ast}}
\nc{\dps}{\dotplus}

\nc{\dth}{d} \nc{\mmbox}[1]{\mbox{\ #1\ }} \nc{\fp}{\mrm{FP}} \nc{\rchar}{\mrm{char}} \nc{\Fil}{\mrm{Fil}} \nc{\Mor}{Mor\xspace} \nc{\gmzvs}{gMZV\xspace} \nc{\gmzv}{gMZV\xspace} \nc{\mzv}{MZV\xspace} \nc{\mzvs}{MZVs\xspace} \nc{\Hom}{\mrm{Hom}} \nc{\id}{\mrm{id}} \nc{\im}{\mrm{im}} \nc{\incl}{\mrm{incl}} \nc{\map}{\mrm{Map}} \nc{\mchar}{\rm char} \nc{\nz}{\rm NZ} \nc{\supp}{\mathrm Supp}
\nc{\mo}{\mathbf o}
\nc{\pl}{\mathfrak{p}}

\nc{\Alg}{\mathbf{Alg}} \nc{\Bax}{\mathbf{Bax}} \nc{\bff}{\mathbf f} \nc{\bfk}{{\bf k}} \nc{\bfone}{{\bf 1}} \nc{\bfx}{\mathbf x} \nc{\bfy}{\mathbf y}
\nc{\base}[1]{\bfone^{\otimes ({#1}+1)}} 
\nc{\Cat}{\mathbf{Cat}} \delete{}
\nc{\detail}{\marginpar{\bf More detail}
    \noindent{\bf Need more detail!}
    \svp}
\nc{\Int}{\mathbf{Int}} \nc{\Mon}{\mathbf{Mon}}
\nc{\rbtm}{{shuffle }} \nc{\rbto}{{Rota-Baxter }} \nc{\remarks}{\noindent{\bf Remarks: }} \nc{\Rings}{\mathbf{Rings}} \nc{\Sets}{\mathbf{Sets}}
\nc{\vwpt}{{Let $V$ be a free $\bfk$-module with a $\bfk$-basis $W$ and let $\Pi$ be a \simple term-rewriting system on $V$ with respect to $W$.}\xspace}

\nc{\BA}{{\Bbb A}} \nc{\CC}{{\Bbb C}} \nc{\DD}{{\Bbb D}} \nc{\EE}{{\Bbb E}} \nc{\FF}{{\Bbb F}} \nc{\GG}{{\Bbb G}} \nc{\HH}{{\Bbb H}} \nc{\LL}{{\Bbb L}} \nc{\NN}{{\Bbb N}} \nc{\KK}{{\Bbb K}} \nc{\QQ}{{\Bbb Q}} \nc{\RR}{{\Bbb R}} \nc{\TT}{{\Bbb T}} \nc{\VV}{{\Bbb V}} \nc{\ZZ}{{\Bbb Z}}

\nc{\cala}{{\mathcal A}} \nc{\calc}{{\mathcal C}} \nc{\cald}{{\mathcal D}} \nc{\cale}{{\mathcal E}} \nc{\calf}{{\mathcal F}} \nc{\calg}{{\mathcal G}} \nc{\calh}{{\mathcal H}} \nc{\cali}{{\mathcal I}} \nc{\call}{{\mathcal L}} \nc{\calm}{{\mathcal M}} \nc{\caln}{{\mathcal N}} \nc{\calo}{{\mathcal O}} \nc{\calp}{{\mathcal P}} \nc{\calr}{{\mathcal R}} \nc{\cals}{{\mathcal S}} \nc{\calt}{{\mathcal T}} \nc{\calw}{{\mathcal W}}
\nc{\calv}{{\mathcal V}}
\nc{\calk}{{\mathcal K}} \nc{\calx}{{\mathcal X}} \nc{\CA}{\mathcal{A}}

\nc{\fraka}{{\mathfrak a}} \nc{\frakA}{{\mathfrak A}} \nc{\frakb}{{\mathfrak b}} \nc{\frakB}{{\mathfrak B}} \nc{\frakC}{{\mathfrak C}}
\nc{\frakD}{{\mathfrak D}} \nc{\frakH}{{\mathfrak H}} \nc{\frakM}{{\mathfrak M}} \nc{\bfrakM}{\overline{\frakM}} \nc{\frakm}{{\mathfrak m}} \nc{\frkP}{{\mathfrak P}}
\nc{\frakN}{{\mathfrak N}} \nc{\frakp}{{\mathfrak p}} \nc{\fraku}{{\mathfrak u}} \nc{\frakv}{{\mathfrak v}}
\nc{\frakQ}{{\mathfrak Q}}\nc{\frakR}{{\mathfrak R}} \nc{\frakS}{{\mathfrak S}}
\nc{\frakx}{{\mathfrak x}} \nc{\ox}{\bar{\frakx}} \nc{\frakX}{{\mathfrak X}} \nc{\fraky}{{\mathfrak y}}
\nc\dop{\delta}
\nc{\Reduce}{{\rm Red}}

\font\cyr=wncyr10 \font\cyrs=wncyr7
\nc{\redt}[1]{\textcolor{red}{#1}}

\nc{\li}[1]{\textcolor{red}{Li:#1}} 
\nc{\lio}[1]{}
\nc{\sz}[1]{\textcolor{green}{sz:#1}}
\nc{\szo}[1]{}
\nc{\xing}[1]{\textcolor{purple}{Xing:#1}}
\nc{\ws}[1]{\textcolor{blue}{{#1}}} 
\nc{\wsc}[1]{\textcolor{blue}{ws:#1}} 
\nc{\wsco}[1]{}
\nc{\wsn}[1]{\textcolor{magenta}{#1}} 

\nc{\medmid}{{\,~{\tiny \longmid}~\,}}
 \nc{\lbar}[1]{\overline{#1}}
\nc{\anf}{\bf $\phi$-NF} \nc{\lto}{\longrightarrow}
\nc{\good}{good\xspace} \nc{\sgood}{super good\xspace}  \nc{\Good}{Good\xspace}
\nc{\brw}{\frakM(Z)} \nc{\irr}{{\rm Irr}} \nc{\pis}{\Pi_S}
\nc{\term}{term\xspace} \nc{\re}[1]{R(#1)} \nc{\sumre}[2]{R^{#1}_{#2}}
\nc{\stars}[2]{#1|_{#2}}\nc{\nbfk}{\bfk^{\times}} \nc{\ord}{{\rm ord}}
\nc{\tpi}{\to_{\Pi}}
\nc{\tpsi}{\to_{\psi}} \nc{\tphi}{\to_{\phi}} \nc{\tvarphi}{\to_{\varphi}}
\nc{\two}{(2)} \nc{\three}{(3)}
\nc{\first}{\alpha} \nc{\second}{\beta}
\nc{\firstx}{\first(x_1, x_2)}\nc{\secondx}{\second(x_1, x_2)}
\nc{\firstu}{\first(u_1, u_2)}\nc{\secondv}{\second(v_1, v_2)}
\nc{\bfirst}{\lbar{\first}} \nc{\bsecond}{\lbar{\second}} \nc{\mstar}{\frakM^\star(X)}
\nc{\xu}{\leq} \nc{\tto}{\leftarrow}

\begin{document}
\title{Averaging algebras, rewriting systems and
Gr\"obner-Shirshov bases}

\author{Xing Gao} \address{School of Mathematics and Statistics,
Key Laboratory of Applied Mathematics and Complex Systems,
Lanzhou University, Lanzhou, 730000, P.R. China}
\email{gaoxing@lzu.edu.cn}

\author{Tianjie Zhang} \address{School of Mathematics and Statistics,
Lanzhou University, Lanzhou, 730000, P.R. China}
\email{907105216@qq.com}

\hyphenpenalty=8000
\date{\today}

\begin{abstract}
In this paper, we study the averaging operator by assigning a rewriting system to it.
We obtain some basic results on the kind of rewriting system we used. In particular, we obtain a sufficient and necessary condition for the confluence. We supply the relationship between rewriting systems and
Gr\"obner-Shirshov bases based on bracketed polynomials.
As an application, we give a basis of the free unitary averaging algebra on a non-empsty set.
\end{abstract}

\delete{
\begin{keyword}

\end{keyword}
}

\maketitle

\tableofcontents

\hyphenpenalty=8000 \setcounter{section}{0}


\section{Introduction}
The averaging operators are generalizations of conditional expectation in probability theory~\mcite{Moy}, and are closely related to Reynolds operators, symmetric operators and Rota-Baxter
operators~\mcite{GM, Tri, Bong}. The study of averaging operators originated from a famous paper on turbulence theory by Reynolds in 1895~\mcite{Re}. The explicit definition of averaging operators
was given in 1930s~\mcite{KF}. Since then there is an extensive literature on averaging operators under various contexts, which can be grouped into two classes. The first one is mainly analytic and for different varieties of averaging algebras; see the references~\mcite{Birk, Fech, Kell, Mill, Moy, Ro0}.
The other class is from an algebraic point of view.
Cao~\mcite{Cao} constructed the free commutative averaging algebras and characterized the
naturally induced Lie algebra structures from averaging operators.
Aguiar proved that the diassociative algebra---the
enveloping algebra of the Leibniz algebra~\mcite{Lo2}---can be obtained from the averaging associative
algebra~\mcite{Agu}. Recently, Guo et al. acquired a relationship between averaging operators and Rota-Baxter operators: the algebraic structures resulted from the actions of
the two operators are Koszul dual to each other~\mcite{GP}. It is worth mentioning that
the Rota-Baxter operator has broad connections with many areas in mathematics
and mathematical physics~\mcite{Bai, Ba, Gub}. In~\mcite{GP}, Guo et al. also
constructed the free nonunitary (noncommutative)
averaging algebra on a non-empty set in terms of a class of bracketed words, by checking the
universal property.
It is natural to construct further the free unitary (noncommutative)
averaging algebra on a non-empty set---our main object of study in the present paper.

Gr\"{o}bner and Gr\"{o}bner-Shirshov bases theory was initiated independently by Shirshov~\mcite{Sh}, Hironaka~\mcite{Hi} and Buchberger~\mcite{Buch}.
It has been proved to be very useful in different branches of mathematics, including commutative algebras and combinatorial algebras, see~\mcite{BC, BCC, BCQ}.
Abstract rewriting system is a branch of theoretical computer science, combining elements of logic, universal algebra, automated theorem proving and functional programming~\mcite{BN, Oh}.
The theories of Gr\"{o}bner-Shirshov bases and rewriting systems are successfully applied to study operators and operator polynomial identities~\mcite{GGSZ, GSZ}.

In the present paper along this line, using the theories of Gr\"{o}bner-Shirshov bases and
rewriting systems, we construct
a basis of the free unitary (noncommutative)
averaging algebra on a non-empty set.
Terminating and confluence
are essential and desirable properties of a rewriting system.
To use the tools of Gr\"{o}bner-Shirshov bases and rewriting systems,
we obtain a sufficient and necessary condition for the confluence of the kind of rewriting system we used.
We supply the relationship between Gr\"{o}bner-Shirshov bases and rewriting systems
based on bracketed polynomials. Applying the method we obtained for checking confluence, we successfully prove that the rewriting system associated to the averaging operator is confluent and then convergent with a suitable order. Let us emphasize that there are a lot of forks in the process of checking confluence. We handle technically most of them in a unified way. These techniques can also be used to study other operators. It is well known that in the category of any given algebraic
structure, the free objects play a central role in study other objects.
Thus as an application, we give a basis of the free unitary (noncommutative) averaging algebra on a non-empty set.

Our characterization of averaging operators in terms of Gr\"{o}bner-Shirshov
bases and rewriting systems reveals the power of this approach. It
would be interesting to further study operators and operator polynomial
identities by making use of the two related theories: Gr\"{o}bner-Shirshov
bases and rewriting systems.

The \emph{layout of the paper} is as follows. In Section~\mref{sec:back}, we first recall the concepts of averaging algebras and free operated algebras. We next recall some necessary backgrounds of Gr\"{o}bner-Shirshov bases and rewriting systems. We obtain some basic results on the kind of
rewriting system we used. In particular, we obtain a sufficient and necessary condition to characterize the confluence (Theorem~\mref{thm:linec}).
We end this section by supplying the relationship between the two powerful tools---Gr\"{o}bner-Shirshov bases and rewriting systems (Theorem~\mref{thm:convsum}).
Section~\mref{sec:average} is devoted to a basis of the free unitary averaging algebra on a non-empty set. In order to achieve this purpose, we assign a rewriting system to the averaging operator (Eq.~(\mref{eq:T1})). We show this rewriting system is convergent (Theorem~\mref{thm:pconf}). We end this section by giving a basis of the free unitary (noncommutative) averaging algebra on a non-empty set (Theorem~\mref{thm:basis}).

Some remark on \emph{notation}. We fix a domain \bfk\ and a non-empty set $X$. Denote by $\nbfk:= \bfk\setminus \{0\}$ the subset of nonzero elements. We denote the \bfk-span of a set $Y$ by $\bfk Y$. For an algebra, we mean a unitary associative noncommutative \bfk-algebra, unless specified otherwise. For any set $Y$, let $M(Y)$ be the free monoid on $Y$ with identity $1$. We use $\sqcup$ for disjoint union.

\section{Gr\"{o}bner-Shirshov bases and rewriting systems}
\mlabel{sec:back}

In this section, we first recall the definition of averaging algebras and characterize free averaging algebras as quotients of free operated algebras. We then recall some backgrounds on Gr\"{o}bner-Shirshov bases and rewriting systems.

\subsection{Free averaging algebras}
\label{sec:freemonoid}

An averaging algebra in the noncommutative context is given
as follows.

\begin{defn}
A linear operator $A$ on a \bfk-algebra $R$ is called an \emph{ averaging operator } if
\begin{equation*}
A(u_1)A(u_2) = A(A(u_1)u_2)= A(u_1A(u_2))\, \text{ for all } u_1, u_2 \in R.
\end{equation*}
A \bfk-algebra $R$ together with an averaging operator $A$ on $R$ is called an \emph{averaging algebra}.
\mlabel{defn:defna}
\end{defn}

To characterize the free averaging algebra, let us recall the free operated algebra~\mcite{BCQ, Gop, Ku}.


\begin{defn}
{\rm An {\bf operated monoid} (resp. {\bf operated $\bfk$-algebra}, resp. {\bf operated $\bfk$-module}) is a monoid (resp. $\bfk$-algebra, resp. $\bfk$-module) $U$ together with a map (resp. $\bfk$-linear map, resp. $\bfk$-linear map) $P_U: U\to U$.
A morphism from an operated monoid\, (resp. $\bfk$-algebra, resp. $\bfk$-module) $(U, P_U)$ to an operated monoid (resp. $\bfk$-algebra, resp. $\bfk$-module) $(V,P_V)$ is a monoid (resp. $\bfk$-algebra, resp. $\bfk$-module) homomorphism $f :U\to V$ such that $f \circ P_U= P_V \circ f$. } \label{de:mapset}
\end{defn}

For any set $Y$, define
$$\lc Y\rc := \{ \lc y\rc \mid y\in Y \},$$
which is a disjoint copy of $Y$. The following is the construction of the free operated monoid on the set $X$, proceeding via the finite stage $\frakM_n(X)$ recursively defined as follows.
Define
$$\frakM_0(X) := M(X)\,\text{ and }\, \frakM_1(X) := M(X \sqcup \lc \frakM_0(X)\rc).$$
Then the inclusion $X\hookrightarrow X \sqcup \lc \frakM_0\rc $ induces a monomorphism
$$i_{0}:\mapmonoid_0(X) = M(X) \hookrightarrow \mapmonoid_1(X) = M(X \sqcup \lc \frakM_0\rc  )$$
of monoids through which we identify $\mapmonoid_0(X) $ with its image in $\mapmonoid_1(X)$.
Suppose that
$\frakM_{ n-1}(X)$ has been defined and the embedding
$$i_{n-2,n-1}\colon  \frakM_{ n-2}(X) \hookrightarrow \frakM_{ n-1}(X)$$
has been obtained for~$n\geq 2$ and consider the case of $n$. Define
\begin{equation*}
 \label{eq:frakn}
 \frakM_{ n}(X) := M \big( X\sqcup\lc\frakM_{n-1}(X) \rc\big).
\end{equation*}
Since~$\frakM_{ n-1}(X) = M \big(X\sqcup \lc\frakM_{ n-2}(X) \rc\big)$ is the free monoid on the set $X\sqcup \lc\frakM_{ n-2}(X) \rc$,
the injection
$$  X \sqcup \lc\frakM_{ n-2}(X) \rc \hookrightarrow
    X \sqcup\lc \frakM_{ n-1}(X) \rc $$
induces a monoid embedding
\begin{equation*}
    \frakM_{ n-1}(X) = M \big( X\sqcup \lc\frakM_{n-2}(X) \rc \big)
 \hookrightarrow
    \frakM_{ n}(X) = M\big( X\sqcup\lc\frakM_{n-1}(X) \rc \big).
\end{equation*}
Finally we define the monoid
$$ \frakM (X):=\dirlim \frakM_{ n}(X) = \bigcup_{n\geq 0} \frakM_{ n}(X).$$
The elements in $ \frakM(X)$ are called \emph{ bracketed words} or
\emph{ bracketed monomials on $X$}. When $X$ is finite, we may also just list its elements, as in $\frakM(x_1,x_2)$ if $X = \{x_1, x_2\}$.
For any $u\in \frakM(X)\setminus \{1\}$, $u$ can be written uniquely as a product: \begin{equation}
u = u_1\cdots u_n, \text{ for some } n\geq 1,\, u_i\in X\sqcup \lc\frakM(X)\rc,\, 1\leq i\leq n.
\mlabel{eq:udecom}
 \end{equation}
The \emph{breadth} of
$u$, denoted by $|u|$, is defined to be $n$. If $u=1$, define $|u|=0$.

Let $\bfk\frakM(X)$ be the free
module with the basis $\frakM(X)$. Using \bfk-linearity, the concatenation product on $\frakM(X)$
can be extended to a multiplication on $\bfk\frakM(X)$, turning $\bfk\frakM(X)$ into a \bfk-algebra.
Define an operator $\lc\,\rc: \frakM(X) \to \frakM(X)$ by assigning
$$ u \mapsto \lc u\rc,\, u\in \frakM(X).$$
By \bfk-linearly, the operator $\lc\,\rc: \frakM(X) \to \frakM(X)$ can be extended to a linear operator $\lc\,\rc: \bfk \frakM(X) \to \bfk \frakM(X)$, turning $(\bfk\frakM(X),\lc\,\rc) $ into an operated \bfk-algebra. The elements in $ \bfk\frakM(X)$ are called
\emph{ bracketed polynomials} or \emph{operated polynomials} on $X$.

\begin{lemma}{\bf \cite[Coro.~3.6, 3.7]{Gop}}
With structures as above,
\begin{enumerate}
\item
the $(\frakM(X),\lc\, \rc)$ together with the natural embedding $i:X \to \frakM(X)$ is the free operated monoid on $X$; and
\label{it:mapsetm}
\item
the $(\bfk\mapm{X},\lc\,\rc)$ together with the natural embedding $i: X \to \bfk\mapm{X}$ is the free operated $\bfk$-algebra on $X$. \label{it:mapalgsg}
\end{enumerate}
\label{pp:freetm}
\end{lemma}

\begin{defn} Let $(R, P)$ be an operated \bfk-algebra.
\begin{enumerate}
\item An element $\phi(x_1,\ldots,x_k)\in \bfk\mapm{X}$ (or $\phi(x_1,\ldots,x_k)=0$) is called an \emph{ operated polynomial identity} (OPI), where $k\geq 1$ and $x_1, \ldots, x_k\in X$.

\item Let $\phi= \phi(x_1,\ldots,x_k)\in \bfk\mapm{X}$ be an OPI. Given any $u_1, \ldots, u_k\in R$, there is a set map $f: x_i \mapsto u_i, 1\leq i\leq k$ and we define
\begin{equation*}
\phi(u_1,\ldots,u_k) := \free{f}(\phi(x_1,\ldots,x_k)),
\label{eq:phibar}
\end{equation*}
where $\free{f}:\bfk\frakM(x_1,\ldots,x_k) \to R$ is the unique morphism of operated algebras that extends the set map $f$, using the universal property of $\bfk\frakM(x_1, \ldots, x_k)$ as the free operated \bfk-algebra on $\{x_1, \ldots, x_k\}$.
Informally, $\phi(u_1,\ldots,u_k)$ is the element of $R$ obtained from $\phi(x_1, \ldots, x_k)$ by replacing each $x_i$ by $u_i$, $1\leq i\leq k$.

\item Let $\Phi\subseteq \bfk\frakM(X)$ be a set of OPIs. We call \emph{$\Phi$ is satisfied} by $R$ if
$$\phi(u_1,\ldots,u_k)=0, \, \forall \phi(x_1,\ldots,x_k)\in \Phi,\, \forall u_1,\ldots,u_k\in R.$$
In this case, we speak that $R$ is a \emph{ $\Phi$-algebra} and $P$ is a \emph{$\Phi$-operator}.

\item Let $S\subseteq \bfk\frakM(X)$ be a set. The \emph{operated ideal} $\Id(S)$ of $\bfk\frakM(X)$ generated by $S$ is the smallest operated ideal containing $S$.
\end{enumerate}
\end{defn}

Let us proceed some examples.

\begin{exam}
The differential operator as an algebraic abstraction of derivation in analysis leads to the differential algebra, which is an algebraic study of differential equations and has been largely successful in many important areas \mcite{Kol, PS, Ri}. The differential operator $d = \lc \,\rc$ fulfils the following OPI
$$\phi(x_1,x_2) = \lc x_1 x_2\rc - \lc x_1\rc x_2 - x_1\lc x_2\rc.$$
\end{exam}

\begin{exam}
The Rota-Baxter operator $P = \lc \,\rc$ of weight $\lambda$ has played important role in mathematics and physics\mcite{Ba, Gub, Ro1}, satisfying the OPI
$$\phi(x_1,x_2) = \lc x_1\rc \lc x_2\rc - \lc x_1\lc x_2\rc\rc - \lc \lc x_1\rc x_2\rc -\lambda \lc x_1x_2\rc,$$
where $\lambda\in \bfk$ is a fixed constant.
\end{exam}

\begin{exam}
From Definition~\mref{defn:defna}, the averaging operator $A = \lc \,\rc$ (noncommutative) is defined by the OPIs
\begin{equation}
\begin{aligned}
\phi(x_1, x_2) &= \lc x_1\rc\lc x_2\rc - \lc \lc x_1\rc x_2\rc,\\
\psi(x_1, x_2) &= \lc x_1\lc x_2\rc\rc - \lc \lc x_1\rc x_2\rc.
\end{aligned}
\mlabel{eq:daver}
\end{equation}
\mlabel{ex:avera}
\end{exam}

\begin{exam}
O. Reynolds~\cite{Re} introduced the concept of Reynolds operators into fluid dynamics, and Kamp\'{e} de F\'{e}riet~\cite{Fe} named it in his study on the various spaces of functions. The Reynolds operator is defined by the OPI
$$\phi(x_1,x_2) =  \lc \lc x_1\rc \lc x_2\rc \rc + \lc x_1\rc \lc x_2\rc - \lc x_1 \lc x_2\rc\rc - \lc \lc x_1\rc x_2\rc.$$
\end{exam}
\begin{defn}
\begin{enumerate}
\item Let $\phi =\phi(x_1,\ldots,x_k)\in \bfk\frakM(X) $ be an OPI with $k\geq 1$. Define
\begin{equation}
 S_\phi(X):= \{\,{\phi}(u_1,\ldots,u_k) \mid u_1,\ldots,u_k\in \frakM(X)\,\}.
\mlabel{eq:genphi}
\end{equation}

\item Let $\Phi$ be a set of OPIs. Define
\begin{equation}
 S_\Phi(X):= \bigcup_{\phi\in \Phi} S_\phi(X).
\mlabel{eq:genPhi}
\end{equation}
\end{enumerate}
\end{defn}

It is well-known that

\begin{prop} {\rm \cite[Prop.~1.3.6]{Coh}} Let $\Phi\subseteq \bfk \frakM(X)$ a set of OPIs. Then the quotient operated algebra $\bfk\mapm{X}/\Id(S_\Phi(X))$ is the free $\Phi$-algebra on $X$.
\label{pp:frpio}
\end{prop}

In particular, we have

\begin{prop} Let $\phi(x_1, x_2)$, $\psi(x_1, x_2)$ defined in Eq.~(\mref{eq:daver}). Then the quotient operated algebra $\bfk\mapm{X}/\Id(S_\phi(X)\cup S_\psi(X))$ is the free averaging algebra on $X$.
\mlabel{pp:freea}
\end{prop}

\subsection{Gr\"{o}bner Shirshov bases} \mlabel{ssec:gsbasis}
In this subsection, we provide some backgrounds on Gr\"obner-Shirshov bases~\cite{BCQ, GGZ, GSZ}.

\begin{defn}
Let $\star$ be a symbol not in $X$ and $X^\star = X \sqcup \{\star\}$.
\begin{enumerate}
\item By a \emph{ $\star$-bracketed word} on $X$, we mean any bracketed word in $\mapm{X^\star}$ with exactly one occurrence of $\star$, counting multiplicities. The set of all $\star$-bracketed words on $X$ is denoted by $\frakM^{\star}(X)$.
\item For $q\in \mstar$ and $u \in  \frakM({X})$, we define $\stars{q}{u}:= q|_{\star \mapsto u}$ to be the bracketed word on $X$ obtained by replacing the symbol $\star$ in $q$ by $u$.

\item For $q\in \mstar$ and $s=\sum_i c_i u_i\in \bfk\frakM{(X)}$, where $c_i\in\bfk$ and $u_i\in \frakM{(X)}$, we define
\begin{equation*}
 q|_{s}:=\sum_i c_i q|_{u_i}\,. \vspace{-5pt}
\end{equation*}

\item A bracketed word $u \in \frakM(X)$ is a \emph{ subword} of another
bracketed word $w \in \frakM(X)$ if $w  = q|_u$ for some $q \in
\frakM^\star(X)$. \mlabel{it:subword}
\end{enumerate}
Generally, with $\star_1,\star_2$ distinct symbols not in
$X$, set $X^{\star 2} := X\sqcup \{\star_1, \star_2\}$.
\begin{enumerate}
\setcounter{enumi}{4}
\item We define an \emph{$(\star_1,
    \star_2)$-bracketed word on $X$} to be a bracketed word in
  $\frakM(X^{\star 2})$ with exactly one occurrence of each of
  $\star_i$, $i=1,2$. The set of all
  $(\star_1, \star_2)$-bracketed words on $X$
  is denoted by $\frakM^{\star_1,\star_2}(X)$.

\item For $q\in \frakM^{\star_1,\star_2}(X)$ and $u_1,u_2 \in
  \bfk \frakM^{\star_1,\star_2}(X)$, we define
$$\stars{q}{u_1, u_2} := \stars{q}{\star_1 \mapsto u_1, \star_2 \mapsto u_2}$$
to be obtained by replacing the letters $\star_i$ in $q$ by $u_i$ for
$i=1,2$.
\end{enumerate}
%
%
 \mlabel{def:starbw}
\end{defn}

\begin{remark}
Recall~\mcite{GSZ} that $\stars{q}{u_1,u_2} = \stars{(\stars{q^{\star_1}}{u_1})}{u_2} = \stars{(\stars{q^{\star_2}}{u_2})}{u_1}$, where $q^{\star_1}$ is viewed as a
$\star_1$-bracketed word on $X\sqcup \{\star_2\}$ and  $q^{\star_2}$ as a
$\star_2$-bracketed word on $X\sqcup \{\star_1\}$.
\end{remark}

We record the following obvious properties of subwords, which will be used later.

\begin{lemma}
Let $u,v,w\in \frakM(X)$.
\begin{enumerate}
\item If $u$ is a subword of $\lc v\rc$, then either $u = \lc v\rc$ or $u$ is a subword of $v$. \mlabel{it:qubra}

\item If $\lc u\rc$ is a subword of $vw$, then either $\lc u\rc$ is a subword of $v$ or $\lc u\rc$ is a subword of $w$.
    \mlabel{it:youbr}
\end{enumerate}
\mlabel{lem:qubra}
\end{lemma}

\begin{proof}
(\mref{it:qubra}) Suppose $u\neq \lc v\rc$. Since $u$ is a subword of $\lc v\rc$, then $\lc v\rc = \stars{q}{u}$ for some $q\in \mstar$ by Definition~\mref{def:starbw}~(\mref{it:subword}).
Since $u\neq \lc v\rc$, it follows that $q\neq \star$. Thus $q = \lc p\rc$ for some $p\in \mstar$ by $\lc v\rc =\stars{q}{u}$. Therefore $\lc v\rc =\stars{q}{u} = \lc \stars{p}{u}\rc$ and so $v = \stars{p}{u}$, as required.

(\mref{it:youbr}). This is followed by the breadth of $\lc u\rc$ is 1.
\end{proof}

The operated ideals in $\bfk\mapm{X}$ can be characterized by $\star$-bracketed words~\cite{BCQ,GSZ}.

\begin{lemma}{\rm (\cite[Lem. 3.2]{GSZ})}
Let $S \subseteq \bfk\mapm{X}$. Then
\begin{equation}\hspace{10pt}\Id(S) = \left\{\, \sum_{i=1}^n c_i q_i|_{s_i} \medmid n\geq 1
{\rm\ and\ } c_i\in \nbfk,
q_i\in \frakM^{\star}(X), s_i\in S {\rm\ for\ } 1\leq i\leq n\,\right\}.
\label{eq:repgen}
\end{equation}
\mlabel{lem:opideal}
\end{lemma}

\begin{defn} A \emph{ monomial order on $\frakM(X)$} is a
well-order $\leq$ on $\frakM(X)$ such that
\begin{equation*}
u < v \Longrightarrow q|_u < q|_v, \quad \forall u, v \in \frakM(X), \forall q \in \frakM^{\star}(X).
%
\end{equation*}
%
\mlabel{de:morder}
\end{defn}

\begin{defn}
Let $s \in \bfk\mapm{X}$ and $\leq$ a linear order on $\frakM(X)$.
\begin{enumerate}
\item Let $s \notin \bfk$. The \emph{leading monomial} of $s$, denoted by $\bar{s}$, is the largest monomial appearing in $s$. The \emph{ leading coefficient of $s$}, denoted by $c_s$, is the coefficient of $\bar{s}$ in $s$.

\label{item:monic}

\item If $s \in \bfk$, we define the \emph{ leading monomial of $s$} to be $1$ and the \emph{ leading coefficient of $s$} to be $c_s= s$.\label{item:scalar}

\item $s$ is called \emph{ monic  with respect to $\leq$} if $s\notin \bfk $ and $c_s=1$. A subset $S \subseteq \bfk\mapm{X}$ is called \emph{ monic with respect to $\leq$} if every $s \in S$ is monic with respect to $\leq$.

\item Define $\re{s}:= c_s\bar{s} -s $. So $s=c_s \bar{s}-\re{s}$.
\mlabel{it:res}
\end{enumerate}
\mlabel{def:irrS}
\end{defn}

We will not need the precise definition of Gr\"{o}bner-Shirshov bases for our construction. So we will not recall it for now and the authors are refereed to~\mcite{BC} and references therein.
Suffices it to say that we need the Composition-Diamond Lemma, the corner stone of Gr\"{o}bner-Shirshov basis theories.

\begin{lemma} {\rm (Composition-Diamond Lemma \cite{BCQ,GSZ})}
Let $\leq$ a monomial order on $\mapm{X}$ and $S \subseteq \bfk\mapm{X}$ monic with respect to $\leq$.
Then the following conditions are equivalent.
\begin{enumerate}
\item $S $ is a Gr\"{o}bner-Shirshov basis in $\bfk\mapm{X}$.
\label{it:cd1}

%

\item $\eta(\Irr(S))$ is a $\bfk$-basis of $\bfk\mapm{X}/\Id(S)$, where
$\eta\!: \bfk\mapm{X} \to \bfk\mapm{X}/\Id(S)$ is the canonical homomorphism of $\bfk$-modules and
\begin{equation}
\irr(S):= \frakM(X) \setminus \{ q|_{\lbar{s}} \mid s\in S \}.
\mlabel{eq:Irrs}
\end{equation}
More precisely as \bfk-modules, $$\bfk \mapm{X}=\bfk\Irr(S)\oplus \Id(S).$$
 \label{it:cd4}
\end{enumerate}
\mlabel{lem:cdl}
\end{lemma}

\subsection{Term-rewriting systems}
In this subsection, we give a method for checking confluence of
\term-rewriting systems. Let us recall some basic notations and results~\mcite{GGSZ}.

\begin{defn}
Let $V$ be a free $\bfk$-module with a given $\bfk$-basis $W$ and $f,g\in V$.
\begin{enumerate}
\item The \emph{support} $\Supp(f)$ of $f$ is the set of monomials (with non-zero coefficients) of $f$. Here we use the convention that $\Supp(0) = \emptyset$. \mlabel{it:suppf}


\item We write $f \dps g$ to indicate that $\Supp(f) \cap \Supp(g) = \emptyset$ and say $f + g$ is a \emph{ direct sum} of $f$ and $g$. If this is the case, we also use $f\dps g$ for the sum $f+ g$.

\item For $w \in \Supp(f)$ with the coefficient $c_w$, we define $R_w(f) := c_w w -f \in V$ and so $f = c_w w \dps (-R_w(f))$.
\end{enumerate}
\mlabel{def:dps}
\end{defn}

\begin{lemma}{\rm \cite[Lem.~2.12]{GGSZ}}
Let $V$ be a free $\bfk$-module with a $\bfk$-basis $W$ and $f,g\in V$.
If $f\dps g$, then $cf\dps dg$ for any $c,d\in \bfk$.%
\mlabel{lem:dsum}
\end{lemma}

\begin{remark}
Using the notation $\dps$, the equation $s=c_s \bar{s}-\re{s}$ in Definition~\mref{def:irrS}~(\mref{it:res})
can be written in more detail as $s=c_s \bar{s} \dps (-\re{s})$.
\mlabel{re:sdps}
\end{remark}

The following is the concept of \term-rewriting systems.

\begin{defn}
Let $V$ be a free $\bfk$-module with a $\bfk$-basis $W$.
A \emph{ \term-rewriting system $\Pi$ on $V$ with respect to $W$} is a binary relation $\Pi \subseteq W \times V$. An element $(t,v)\in \Pi$ is called a \emph{(term) rewriting rule} of $\Pi$, denoted by $t\to v$. The \term-rewriting system $\Pi$ is called \emph{ \simple} if $t \dps v$ for all $t\to v\in \Pi$.
\mlabel{defn:trs}
\end{defn}

\begin{remark}
Now we explain the requirement that the \term-rewriting system $\Pi$ is simple. Suppose $\Pi$ is not simple. Then by Definition~\mref{defn:trs}, there is a rewriting rule $t\to v$ such that $t\in \Supp(v)$. Assume $v = ct \dps (-R_t(v))$ for some $c\in \nbfk$. Then
$$t \to_\Pi v = ct \dps (-R_t(f)) \to_\Pi cv - R_t(v) = c^2 t\dps (-c-1)R_t(v) \to_\Pi \cdots .$$
So as long as $c$ is not a nilpotent element, $\Pi$ is not terminating. In the remainder of this
paper, we always assume that the \term-rewriting system is simple, unless specified otherwise.
\end{remark}

\begin{defn}
Let $V$ be a free $\bfk$-module with a $\bfk$-basis $W$, $\Pi$ a simple \term-rewriting system on $V$ with respect to $W$ and $f,g\in V$.
\begin{enumerate}
\item We speak that $f$ \emph{ rewrites} to $g$ \emph{ in one-step}, denoted by $f \to_\Pi g$ or $\tvarrow{f}{g}{\Pi}$, if
$$ f = c_t t\dps (-R_t(f))\,\text{ and }\,g= c_t v - R_t(f)\, \text{ for some } c_t\in \nbfk\,\text{ and }\, t\to v\in \Pi.$$
\mlabel{it:Trule}

\item The reflexive-transitive closure of the binary relation $\rightarrow_\Pi$ on $V$ is denoted by $\astarrow_\Pi$. If $f \astarrow_\Pi g$ (resp. $f \not\astarrow_\Pi g$), we speak that \emph{ $f$ rewrites (resp. doesn't rewrite ) to $g$ with respect to $\Pi$}. \mlabel{it:rtcl}

\item We call $f$ and  $g$ are \emph{ joinable}, denoted by $f \downarrow_\Pi g$, if there exists $h \in V$ such that $f \astarrow_\Pi h$ and $g \astarrow_\Pi h$.
    \mlabel{it:joinfg}

\item We say $f$ \emph{a normal form} if no more rewriting rules can apply.
\end{enumerate}
\mlabel{def:ARSbasics}
\end{defn}

\begin{remark} Let $f,g\in V$.
\begin{enumerate}
\item By Definition~\mref{def:ARSbasics}~(\mref{it:rtcl}), $f\astarrow_\Pi f$ and
\begin{equation*}
f\astarrow g \Longleftrightarrow f=: f_0 \tpi f_1 \tpi \cdots \tpi f_n:= g \,\text{ for some } n\geq 0, f_i\in V, 0\leq i\leq n.
\end{equation*}
   \mlabel{it:itself}

\item If $f\astarrow_\Pi g$, then $f\downarrow_\Pi g$ by $g\astarrow_\Pi g$.
In particular, $f\downarrow_\Pi f$ by $f\astarrow_\Pi f$.
\mlabel{it:tdown}

\end{enumerate}
\mlabel{re:fgtri}
\end{remark}

The following definitions are adapted from abstract rewriting systems \mcite{BN, BKVT}.

\begin{defn} Let $V$ be a free $\bfk$-module with a $\bfk$-basis $W$, $\Pi$ a simple \term-rewriting system on $V$ with respect to $W$.
 \begin{enumerate}
\item $\Pi$ is \emph{ terminating} if there is no infinite chain of one-step rewriting \vspace{-3pt}$$f_0 \rightarrow_\Pi f_1 \rightarrow_\Pi f_2 \cdots \quad.\vspace{-3pt}$$

\item $f\in V$ is \emph{locally confluent} if for every local fork $(h \prescript{}{\Pi}{\tto} f\to_\Pi g)$, we have $g\downarrow_\Pi h$.

\item $f\in V$ is \emph{confluent} if for every fork $(h \prescript{}{\Pi}{\lastarrow} f\astarrow_\Pi g)$, we have $g\downarrow_\Pi h$.

\item $\Pi$ is \emph{locally confluent (resp. confluent)} if every $f\in V$ is locally confluent (resp. confluent).

\item $\Pi$ is \emph{ convergent} if it is both terminating and confluent. \mlabel{it:dconv}
\end{enumerate}
\mlabel{defn:ARS}
\end{defn}

A well-known result on rewriting systems is Newman's Lemma.

\begin{lemma}{\rm (\cite[Lem. 2.7.2]{BN})}
 A terminating rewriting system is confluent if and only if it is locally confluent. \label{lem:newman}
 \end{lemma}

The following result will be used later.

\begin{lemma} {\rm (\cite[Thm. 2.20]{GGSZ})}
Let $V$ be a free $\bfk$-module with a $\bfk$-basis $W$ and
$\Pi$ a simple \term-rewriting system on $V$ with respect to $W$.
If\, $\Pi$ is confluent, then, for all $m \geq 1$ and
$f_1, \dots, f_m, g_1, \dots, g_m \in V$, $$f_i \downarrow_\Pi g_i  \quad (1 \leq i \leq m),
{\rm\ and\ } \sum_{i=1}^m g_i = 0 \implies \left(\sum_{i=1}^m f_i \right) \astarrow_\Pi 0.$$
%
%
%
%
\mlabel{lem:red}
\end{lemma}

\begin{remark}
If $\Pi$ is confluent and $f \downarrow_\Pi g$, then $f-g\astarrow_\Pi 0$
by $-g \downarrow_\Pi -g$ and Lemma~\mref{lem:red}.
\mlabel{re:tozero}
\end{remark}

The following is a concept strong than locally confluence and similar to Buchberger's $S$-polynomials.

\begin{defn}
Let $V$ be a free $\bfk$-module with a $\bfk$-basis $W$, $\Pi$ a simple \term-rewriting system on $V$ with respect to $W$.
\begin{enumerate}
\item \emph{ A local base-fork} is a fork $(cv_1 \prescript{}{\Pi}{\tto}\ ct\to_\Pi cv_2)$, where $c \in \nbfk$ and $t\to v_1, t\to v_2 \in \Pi$ are rewriting rules.

\item The \term-rewriting system $\Pi$ is called \emph{locally base-confluent} if for every local base-fork $(cv_1 \prescript{}{\Pi}{\tto}\ ct\to_\Pi cv_2)$, we have $c (v_1 - v_2) \astarrow_\Pi 0$.

\item $\Pi$ is \emph{compatible} with a linear order $\leq$ on $W$
if $\lbar{v} < t$ for each $t\to v \in \Pi$. \mlabel{it:compor}

\end{enumerate}
\mlabel{defn:bfork}
\end{defn}

\begin{lemma}{\rm (\cite[Lem.~2.22]{GGSZ})} Let $V$ be a free $\bfk$-module with a $\bfk$-basis $W$ and let $\Pi$ be a simple \term-rewriting system on $V$ which is compatible with a well order $\leq$ on $W$. If $\Pi$ is locally base-confluent, then it is locally confluent.
\mlabel{lem:ltcon}
\end{lemma}

The following concept is followed from general abstract rewriting systems~\cite[Def.~1.1.6]{BKVT}.

\begin{defn}
Let $V$ be a free $\bfk$-modules with a $\bfk$-basis $W$ and let $\Pi$ be a \simple term-rewriting system on $V$ with respect to $W$. Let $Y\subseteq W$ and $\Pi_{\bfk Y}:= \Pi \cap(Y\times \bfk Y)$. We call $\Pi_{\bfk Y}$ a \emph{sub-term-rewriting system} of $\Pi$ on $\bfk Y$ with respect to $Y$, denoted by $\Pi_{\bfk Y} \leq \Pi$, if
$\bfk Y$ is closed under $\Pi$, i.e., for any $f\in \bfk Y$ and any $g\in V$,
$f\to_{\Pi} g$ implies $g\in \bfk Y$.
\mlabel{defn:subre}
\end{defn}

\begin{remark}
Since $\Pi$ is simple, $\Pi_{\bfk Y}$ is also simple. Indeed, let $t\to v\in \Pi_{\bfk Y}$ be a rewriting rule with
$t\in Y$ and $v\in \bfk Y$. Then $t\to v\in \Pi$ by $\Pi_{\bfk Y}\subseteq \Pi$. Since $\Pi$ is simple, we have $t\notin \Supp(v)$ by Definition~\mref{defn:trs} and so $\Pi_{\bfk Y}$ is simple.
\end{remark}

We record the following properties.

\begin{lemma}
Let $V$ be a free $\bfk$-module with a $\bfk$-basis $W$, and let
$\Pi$ be a simple \term-rewriting system on $V$ with respect to $W$.

\begin{enumerate}
\item If $t\in \Supp(cf)$ with $t\in W$, $c\in \nbfk$ and $f\in V$, then $t\in \Supp(f)$.
\mlabel{it:tinf}

\item If $cf \tpi g$ with $c\in \nbfk$ and $f, g\in V$, then $g = cg'$ for some $g'\in V$.
\mlabel{it:geqcg}

\item If $cf = 0$ with $c\in \nbfk$ and $f\in V$, then $f=0$. \mlabel{it:feq0}
\item For $c\in \nbfk$ and $f, g \in V$ with $f\neq g$, $f\tpi g$ if and only if $cf\tpi cg $.
\mlabel{it:quc}
\end{enumerate}
\mlabel{lem:cform}
\end{lemma}

\begin{proof}
(\mref{it:tinf}) Suppose for a contrary that $t\notin\Supp(f)$. Since $W$ is a \bfk-basis of $V$, by Definition~\mref{def:dps}~(\mref{it:suppf}), we may write $f = \sum_i c_i w_i$,
where each $c_i\in \nbfk$ and $w_i\in W\setminus\{t\}$. Then $cf = \sum_i cc_i w_i$. Since $w_i\neq t$
for each $i$, we have $t\notin \Supp(cf)$, a contradiction.

(\mref{it:geqcg}) Suppose $cf\overset{(t,v)}{\lto_{\Pi}} g$ for some $t\to v\in \Pi$. Then $t\in \Supp(cf)$ and so $t\in \Supp(f)$ by Item~(\mref{it:tinf}). Write
$f = c_tt\dps (-R_t(f))$ with $c_t\in\nbfk$. Then by Lemma~\mref{lem:dsum},
$$cf = cc_tt \dps (-cR_t(f)) \overset{(t,v)}{\lto_{\Pi}}  cc_tv -cR_t(f) = c(c_tv - R_t(f))= g,$$ as required.

(\mref{it:feq0}) Since $W$ is a \bfk-basis of $V$, we may write $f = \sum_i c_i w_i$ with
$c_i\in \bfk$ and $w_i\in W$ for each $i$. Then $cf = \sum_i cc_i w_i = 0$ and so $cc_i = 0$ for each $i$. Since $\bfk$ is a domain by our hypothesis and $c\neq 0$, we have $c_i = 0$ for each $i$, that is, $f = 0$.

(\mref{it:quc}) Suppose $f\overset{(t,v)}{\lto_{\Pi}} g$ for some $t\to v\in \Pi$. By Definition~\mref{def:ARSbasics}~(\mref{it:Trule}), we may write
$$f = dt\dps (-R_t(f))\, \text{ and }\, g = dv - R_t(f)\, \text{ for some } d\in \nbfk.$$
Then by Lemma~\mref{lem:dsum},
$$cf = cdt\dps (-cR_t(f))\, \text{ and }\, cg = cdv - cR_t(f)$$
and so $cf \overset{(t,v)}{\lto_{\Pi}} cg$.
Conversely, suppose $cf \overset{(t,v)}{\lto_\Pi} cg$ for some $t\to v\in \Pi$.
Then $t\in \Supp(cf)$ and so $t\in \Supp(f)$ by Item~(\mref{it:tinf}). Write $f = c_tt\dps (-R_t(f))$ with $c_t\in \nbfk$.
Then from Lemma~\mref{lem:dsum}, $$cf = cc_tt\dps (-cR_t(f)) \overset{(t,v)}{\lto_\Pi} cc_tv - cR_t(f) =cg.$$
Since $c\in\nbfk$, we get
$c_tv - R_t(f) = g$ by Item~(\mref{it:feq0}) and so $f\tpi g$.
\end{proof}

\begin{lemma}
Let $V$ be a free $\bfk$-module with a $\bfk$-basis $W$, and let
$\Pi$ be a simple \term-rewriting system on $V$ with respect to $W$. Let $f, g \in V$ and $c\in \nbfk$. Then $f\astarrow_\Pi g$ if and only if $cf\astarrow_\Pi cg $.
\mlabel{lem:cons}
\end{lemma}

\begin{proof}
($\Rightarrow$) If $f=g$, then $cf =cg$ and $cf \astarrow_\Pi cg$ by Remark~\mref{re:fgtri}~(\mref{it:itself}). Suppose $f\neq g$. Let $n\geq 1$ be the least number such that $f$ rewrites to $g$ in $n$ steps.
Then
\begin{equation}
f=f_0\tpi f_1\tpi \cdots \tpi f_n=g
\mlabel{eq:f0n}
\end{equation}
$\text{ for some distinct } f_i\in V$, $0\leq i\leq n$ and so by Lemma~\mref{lem:cform}~(\mref{it:quc}),
\begin{equation}
 cf=cf_0\tpi cf_1\tpi \cdots \tpi cf_n=cg.
\mlabel{eq:cf0n}
\end{equation}
Hence $cf\astarrow_\Pi cg$.

($\Leftarrow$) If $cf = cg$, then $f=g$ by Lemma~\mref{lem:cform}~(\mref{it:feq0}) and
so $f\astarrow_\Pi g$ by Remark~\mref{re:fgtri}~(\mref{it:itself}). Suppose $cf\neq cg$.
Let $n\geq 1$ be the least number such that $cf$ rewrites to $cg$ in $n$ steps.
Then by Lemma~\mref{lem:cform}~(\mref{it:geqcg}), Eq~(\mref{eq:cf0n}) holds for some distinct $cf_i\in V$, $0\leq i\leq n$. Using Lemma~\mref{lem:cform}~(\mref{it:feq0}), $f_i\in V$ are distinct for $0\leq i\leq n$.
From Lemma~\mref{lem:cform}~(\mref{it:quc}), Eq.~(\mref{eq:f0n}) is valid and so $f\astarrow_\Pi g$.
\end{proof}

\begin{theorem}
Let $V$ be a free $\bfk$-module with a $\bfk$-basis $W$ and let $\Pi$ be a \simple terminating term-rewriting system on $V$ with respect to $W$. Suppose $\leq$ is a well-order on $W$ compatible with $\Pi$. Then
$\Pi$ is confluent if and only if $w$ is locally confluent for any $w\in W$.
\mlabel{thm:linec}
\end{theorem}

\begin{proof}
($\Rightarrow$) Since $\Pi$ is confluent, $\Pi$ is locally confluent by Definition~\mref{defn:ARS}, that is, every element in $V$ is locally confluent. From $W\subseteq V$, $w$ is locally confluent for any $w\in W$.

($\Leftarrow$) To show $\Pi$ is confluent, it is enough to show $\Pi$ is locally confluent by Lemma~\mref{lem:newman}. In view of Lemma~\mref{lem:ltcon}, we are left to prove that $\Pi$ is locally base-confluent, that is, for any local base-fork
$(cv_1 \prescript{}{\Pi}{\tto}\ cw\to_\Pi cv_2)$, we have $cv_1 - cv_2 \astarrow_\Pi 0$. Suppose for a contrary that $\Pi$ is not locally base-confluent. Then the set
\begin{align*}
\frakC = \{ w\in W \medmid &\text{ there is a local fork base-fork }(cv_1 \prescript{}{\Pi}{\tto}\ cw\to_\Pi cv_2) \\
&\text{ for some $c\in \nbfk, v_1,v_2\in V$ such that } cv_1 - cv_2 \not\astarrow_\Pi 0\}
\end{align*}
is non-empty. Since $\leq$ is a well-order, $\frakC$ has the least element $w$ with respect to $\leq$.
Thus there is a local base-fork
\begin{equation}
(cv_1 \prescript{}{\Pi}{\tto}\ cw\to_\Pi cv_2)\,\text{ with } w\to v_1, w\to v_2\in \Pi
\mlabel{eq:wfork}
\end{equation}
such that
\begin{equation}
cv_1 -cv_2\not\astarrow_\Pi 0\,\text{ for some }\, c\in \nbfk,v_1, v_2\in V.
\mlabel{eq:contr}
\end{equation}
Let
\begin{equation}
Y:= \{y\in W \mid y < w\}\,\text{ and }\, \Pi_{\bfk Y} = \Pi \cap (Y\times \bfk Y).
\mlabel{eq:defny}
\end{equation}
Since $\leq$ is compatible with $\Pi$, we have $\Supp(v_1), \Supp(v_2) \subseteq Y$ and so $Y\neq \emptyset$. Furthermore,  $\Pi_{\bfk Y} \leq \Pi$ is a sub-\term-rewriting system of $\Pi$. Indeed, let $f\astarrow_\Pi g$ with $f\in \bfk Y$, since $\leq$ is compatible with $\Pi$, we get $\bar{g}\leq \bar{f} < w$ and so $g\in \bfk Y$. Thus $\Pi_{\bfk Y}$ is closed under $\Pi$ and so $\Pi_{\bfk Y} \leq \Pi$ by Definition~\mref{defn:subre}.

For any local base-fork $(du_1 \prescript{}{\Pi_{\bfk Y}}{\tto}\ dy\to_{\Pi_{\bfk Y}} d u_2)$ of $\Pi_{{\bfk Y}}$ with $d\in \nbfk$, $y\in Y$ and $u_1, u_2\in \bfk Y$, it induces a local base-fork $(du_1 \prescript{}{\Pi}{\tto}\ dy\to_{\Pi} d u_2)$ by $\Pi_{\bfk Y} \subseteq \Pi$. Since $y\in Y$, we have $y<w$ and $y\notin \frakC$ by the minimality of $w$. So $du_1 - du_2\astarrow_{\Pi} 0$ by the definition of $\frakC$.
Claim
\begin{equation}
f\astarrow_\Pi g \Longrightarrow f\astarrow_{\Pi_{\bfk Y}} g \,\text{ for }\, f,g\in \bfk Y.
\mlabel{eq:cfgin}
\end{equation}
Since $du_1 - du_2\in \bfk Y$ by $u_1, u_2\in \bfk Y$, we have
$du_1 - du_2\astarrow_{\Pi_{\bfk Y}} 0$ by the Claim. Thus $\Pi_{\bfk Y}$ is locally base-confluent and so is locally confluent by Lemma~~\mref{lem:ltcon}. Since $\Pi$ is terminating
and $\Pi_{\bfk Y}\leq \Pi$, $\Pi_{\bfk Y}$ is terminating. Therefore $\Pi_{\bfk Y}$ is confluent by Lemma~\mref{lem:newman}.

For the local fork in Eq.~(\mref{eq:wfork}), it induces a local fork $(v_1 \prescript{}{\Pi}{\tto}\ w\to_\Pi v_2)$ by Lemma~\mref{lem:cform}~(\mref{it:quc}).
Since $w\in W$ is confluent by our hypothesis, it follows that $v_1 \downarrow_\Pi v_2$.
So there is $u\in V$ such that $v_1\astarrow_\Pi u$ and  $v_2\astarrow_\Pi u$ by Definition~\mref{def:ARSbasics}~(\mref{it:joinfg}). From Lemma~\mref{lem:cons},
$$cv_1\astarrow_\Pi cu\,\text{ and }\, cv_2\astarrow_\Pi cu.$$
From $cv_1\in \bfk Y$ and $\Pi_{\bfk Y}\leq \Pi$ is closed under $\Pi$, we have $cu\in \bfk Y$. So by the Claim of Eq.~(\mref{eq:cfgin}),
$$cv_1\astarrow_{\Pi_{\bfk Y}} cu\,\text{ and }\, cv_2\astarrow_{\Pi_{\bfk Y}} cu. $$
This means that $cv_1 \downarrow_{\Pi_{\bfk Y}} cv_2$. Since $\Pi_{\bfk Y}$ is confluent, $cv_1 - cv_2\astarrow_{\Pi_{\bfk Y}} 0 $ by Remark~\mref{re:tozero}. Hence $cv_1 - cv_2\astarrow_{\Pi} 0$ by $\Pi_{\bfk Y}\subseteq \Pi$, contradicting Eq.~(\mref{eq:contr}). We are left to prove the Claim.
\end{proof}

\begin{proof}[proof of Claim]
We want to show Eq.~(\mref{eq:cfgin}). Suppose $f\astarrow_\Pi g$ with $f,g\in \bfk Y$.
If $f=g$, then $f\astarrow_{\Pi_{\bfk Y}} g$ by Remark~\mref{re:fgtri}~(\mref{it:itself}).
Assume $f\neq g$ and let $n\geq 1$ be least number such that
$$f=: f_0 \tpi f_1 \tpi \cdots \tpi f_n:= g \,\text{ with } f_i\in V \text{ are distinct}, 0\leq i\leq n.$$
Since $f_0 =f\in \bfk Y$ and $\Pi$ is compatible with $\leq$, we have $f_i\in \bfk Y$ for $0\leq i\leq n$. We prove the Claim by induction on $n\geq 1$. For the initial step of $n=1$, suppose
$f=f_0 \overset{(t,v)}{\to}_\Pi f_1 =g$ for some $t\to v\in \Pi$. Then $t\in \Supp(f)\subseteq Y$.
This follows that $t< w$ by Eq.~(\mref{eq:defny}). Since $\Pi$ is compatible with $\leq$, we have $\lbar{v} < t < w$ and so $v\in \bfk Y$. Thus $t\to v\in Y\times \bfk Y$ and so $t\to v\in \Pi \cap (Y\times \bfk Y) = \Pi_{\bfk Y}$. This implies that $f=f_0 \overset{(t,v)}{\lto}_{\Pi_{\bfk Y}} f_1 = g$ by $f_0,f_1\in \bfk Y$ and $f_0\overset{(t,v)}{\lto}_\Pi f_1$. For the induction step, we have
$f =f_0 \astarrow_{\Pi_{\bfk Y}} f_1$ and $f_1 \astarrow_{\Pi_{\bfk Y}} f_n =g$ by induction hypothesis and so $f\astarrow_{\Pi_{\bfk Y}} g$, as required.
\end{proof}

\subsection{Term-rewriting systems and Gr\"{o}bner-Shirshov bases}

In this subsection, we supply the relationship between Gr\"{o}bner-Shirshov bases and
term-rewriting systems based on bracketed polynomials. A \term-rewriting system can be assigned
to a given set $S$ of OPIs~\mcite{GGSZ}.

\begin{defn}
Let $\leq$ be a linear order on $\frakM(X)$ and $S \subseteq \bfk\mapm{X}$ monic with respect to $\leq$. Define a \term-rewriting system associated to $S$ as
\begin{equation}
\Pi_{S} := \{\, q|_{\bar{s}}\to  q|_{R(s)} \mid  s=\lbar{s}\dps (-\re{s}) \in S, ~q \in \frakM^\star(X)\,\} \subseteq \mapm{X} \times \bfk\mapm{X}.
\mlabel{eq:rwS}
\end{equation}
\label{def:redsys}
\end{defn}

\vspace{-0.2in}

For notation clarify, we denote $\to_{\Pi_S}$ (resp. $\astarrow_{\Pi_S}$, resp. $\downarrow_{\Pi_S}$) by $\to_S$ (resp. $\astarrow_S$, resp. $\downarrow_{S}$).
In more detail when a specific $s\in S$ is used in one step rewriting, we replace $\to_{S}$ by
$\to_{s}$.
If $\leq$ is a monomial order on $\frakM(X)$, we have $\bar{\stars{q}{R(s)}} = \stars{q}{\bar{R(s)}}< \stars{q}{\bar{s}}$ by $\bar{R(s)} < \bar{s}$. So $\Pi_S$ is compatible with $\leq$ in the sense in Definition~\mref{defn:bfork}~(\mref{it:compor}).

\begin{remark} Let $f,g\in \bfk \frakM(X)$.
\begin{enumerate}
\item If $f\to_S g$, then we can write
$f = c\stars{q}{\bar{s}} \dps f'$ and $g = c\stars{q}{R(s)} +f' $
for some $c\in \nbfk$, $q\in\mstar$, $s\in S$ and $f'\in \bfk\frakM(X)$ by Definition~\mref{def:ARSbasics}~(\mref{it:Trule}). So $f-g = c \stars{q}{\lbar{s} - R(s)} = c \stars{q}{s} \in \Id(S)$
by Lemma~\mref{lem:opideal}.
\mlabel{it:oins}

\item If $f\astarrow_S g$, then $f=: f_0 \to_S f_1 \to_S \cdots
\to_S f_n:= g \,\text{ for some } n\geq 0, f_i\in \bfk \frakM(X), 0\leq i\leq n.$
If $n=0$, then $f=g$ and $f-g\in \Id(S)$. If $n\geq 1$, then by Item~(\mref{it:oins}),
$$f-g = (f_0 -f_1) + (f_1 -f_2) +\cdots + (f_{n-1} -f_{n}) \in \Id(S).$$
\mlabel{it:lins}
\end{enumerate}
\mlabel{re:inideal}
\end{remark}

\begin{lemma}
If $u\dps v$, then $\stars{q}{u}\dps \stars{q}{v}$ for any $q\in \mstar$ and $u,v\in \bfk\frakM(X)$.
\mlabel{lem:qsum}
\end{lemma}

\begin{proof}
Write $u = \sum_i c_i u_i$ and $v =\sum_j d_j v_j$, where each $c_i,d_j\in \nbfk$ and $u_i, v_j\in \frakM(X)$. Then
$$\stars{q}{u} = \sum_i c_i \stars{q}{u_i}\,\text{ and }\, \stars{q}{v} = \sum_j d_j \stars{q}{v_j}.$$
Suppose for a contrary that $\stars{q}{u}\dps \stars{q}{v}$ fails. Then $\stars{q}{u_i} = \stars{q}{v_j}$ by Definition~\mref{def:dps} for some $i,j$. This implies that $u_i = v_j\in \Supp(u) \cap \Supp(v)$, contradicting that $u\dps v$.
\end{proof}

The following results are characterized in~\mcite{GG}.
For completeness, we record the proof here.

\begin{lemma}
Let $\leq$ be a linear order on $\frakM(X)$ and $S \subseteq \bfk\frakM(X)$ monic with respect to $\leq$.
\begin{enumerate}
\item If $\,\pis$ is confluent, then, $u\in \Id(S)$ if and only if $u\astarrow_{\pis} 0$. \mlabel{it:idealto0}

\item  If $\,\pis$ is confluent, then $\Id(S)\cap \bfk \Irr(S) =0$. \mlabel{it:confcap}

\item If $\,\pis$ is terminating and $\Id(S)\cap \bfk \Irr(S) =0$, then $\pis$ is confluent. \mlabel{it:capconf}

\item If $\,\pis$ is terminating, then $\bfk\frakM(X) = \Id(S) +  \bfk \Irr(S)$, \mlabel{it:tsum}
\end{enumerate}
where $\irr(S)= \frakM(X) \setminus \{ q|_{\lbar{s}} \mid s\in S \}$.
\mlabel{lem:useful}
\end{lemma}

\begin{proof} Note that $\bfk \irr(S)$ is precisely the set of
normal forms of $\Pi_S$.

(\mref{it:idealto0}) If $u\astarrow_{\pis} 0$, then $u\in \Id(S)$ by Remark~\mref{re:inideal}~(\mref{it:lins}).
Conversely, let $u\in \Id(S)$. By Eq.~(\mref{eq:repgen}), we have
$$ u = \sum_{i=1}^n c_i q_i|_{s_i}, \, \text{ where }\, c_i\in \nbfk, s_i\in S, q_i\in \frakM^\star(X), 1\leq i\leq n.$$
For each $s_i =\lbar{s_i}\dps (-\re{s_i})$ with $1\leq i\leq n$,
it follows from Lemmas~\mref{lem:dsum} and~\mref{lem:qsum} that
$$c_i\stars{q_i}{s_i} = c_i\stars{q_i}{\lbar{s_i}} \dps (-c_i\stars{q_i}{\re{s_i}})
 \to_{\pis} c_i\stars{q_i}{\re{s_i}} - c_i\stars{q_i}{\re{s_i}}=0 \,
 \text{ and so }\, c_iq_i|_{s_i}\downarrow_{\pis} 0$$
 by Remark~\mref{re:fgtri}~(\mref{it:tdown}).
Since $\pis$ is confluent,
 $u = \sum_{i=1}^n c_i q_i|_{s_i} \astarrow_{\pis} 0$ by Lemma~\mref{lem:red}.

(\mref{it:confcap}) If $\Id(S)\cap \bfk \Irr(S)\neq 0$, let $0\neq w\in \Id(S)\cap \bfk \Irr(S)$. Since $w\in \bfk \Irr(S)$, $w$ is of normal form. On the other hand, from $w\in \Id(S)$ and Item~(\mref{it:idealto0}), we have $w\astarrow_{\pis} 0 $. So $w$ has two normal forms $w$ and $0$, contradicting that $\pis$ is confluent.

(\mref{it:capconf}) Suppose for a contrary that $\pis$ is not confluent. Since $\pis$ is terminating, there is $w\in \bfk\frakM(X)$ such that $w$ has two distinct normal forms, say $u$ and $v$. Thus $u,v\in \bfk\Irr(S)$ and so $u-v\in \bfk\Irr(S)$. Since $w\astarrow_\Pi u$ and $w\astarrow v_\Pi$, we have $w-u, w-v\in \Id(S)$ by Remark~\mref{re:inideal}~(\mref{it:lins}).
Hence $0\neq u-v\in\Id(S) \cap \bfk\Irr(S)$, a contradiction.

(\mref{it:tsum}) Let $w\in \bfk\frakM(X)$, since $\pis$ is terminating, $w$ has a normal form $u\in \bfk \Irr(S)$ and $w\astarrow_\Pi u$. From Remark~\mref{re:inideal}~(\mref{it:lins}), we have $w-u\in \Id(S)$ and so $w\in \Id(S) + \bfk \Irr(S)$.
\end{proof}

\begin{theorem}
Let $\leq$ be a monomial order on $\frakM(X)$ and $S \subseteq \bfk\frakM(X)$ monic with respect to $\leq$. Then the followings are equivalent.
\begin{enumerate}
\item $\pis$ is convergent. \mlabel{it:sconv}

\item $\pis$ is confluent. \mlabel{it:sconf}

\item $\Id(S) \cap \bfk \Irr(S) = 0$. \mlabel{it:scap}

\item $\Id(S) \oplus \bfk \Irr(S) =\bfk \frakM(X)$. \mlabel{it:ssum}

\item $S$ is a Gr\:{o}bner-Shirshov basis in $\bfk\frakM(X)$, \mlabel{it:sgsbasis}

\end{enumerate}
where $\irr(S)= \frakM(X) \setminus \{ q|_{\lbar{s}} \mid s\in S \}$.
\mlabel{thm:convsum}
\end{theorem}

\begin{proof}
Since $\leq $ is a monomial order on $\frakM(X)$, $\pis$ is terminating~\mcite{GGSZ}. So Item~(\mref{it:sconv}) and Item~(\mref{it:sconf}) are equivalent.
The equivalence of Item~(\mref{it:sconf}) and Item~(\mref{it:scap})
is followed from Items~(\mref{it:confcap}) and~(\mref{it:capconf}) in Lemma~\mref{lem:useful}.

Clearly, Item~(\mref{it:ssum}) implies Item~(\mref{it:scap}). The converse is employed Item~(\mref{it:tsum}) in Lemma~~\mref{lem:useful}. At last, the equivalence of Item~(\mref{it:ssum}) and Item~(\mref{it:sgsbasis}) is obtained from Lemma~\mref{lem:cdl}.
\end{proof}

\section{A basis of the free averaging algebra}
\mlabel{sec:average}
In this section, we give a basis of the free averaging algebra. We begin with a lemma.

\begin{lemma}
Let $S \subseteq \bfk \frakM(X)$, $q\in \frakM^\star(X)$ and
$\leq$ a linear order on $\frakM(X)$. Then
\begin{enumerate}

\item If $u \astarrow_S v$ with $u, v\in \bfk \frakM(X)$, then $\stars{q}{u} \astarrow_S \stars{q}{v}$. \mlabel{it:qs} \mlabel{it:qphi}

\item If $u\downarrow_S v$, then $\stars{q}{u} \downarrow_S \stars{q}{v}$. \mlabel{it:qsf}
\end{enumerate}
\mlabel{lem:qs}
\end{lemma}

\begin{proof}
(\mref{it:qphi})
If $u =v$, then $\stars{q}{u} = \stars{q}{v}$ and $\stars{q}{u} \astarrow_S\stars{q}{v}$ by Remark~\mref{re:fgtri}~(\mref{it:itself}).
Suppose $u\neq v$. Let $m\geq 1$ be the least number such that $u$ rewrites to $v$ in $m$ steps. We prove the result by induction on $m$. For the initial step $m=1$, since $u \to_S v$, we may write
$$u = c\stars{p}{\lbar{s}} \dps u'\,\text{ and }\, v = c\stars{p}{\re{s}} + u'\, \text{ for some }\, c\in \nbfk, s\in S, p\in \frakM^\star(X), u'\in \bfk \frakM(X).$$
Then from Lemma~\mref{lem:qsum},
$$\stars{q}{u} = c\stars{(\stars{q}{p})}{\lbar{s}} \dps \stars{q}{u'} \to_S
 c\stars{(\stars{q}{p})}{\re{s}} +  \stars{q}{u'} = \stars{q}{ c\stars{p}{\re{s}} + u' } = \stars{q}{v}.$$
Assume the result is true for $m \leq n$ and consider the case of $m = n+1\geq 2$.
Then we can write $u \to_S w \astarrow_S v$ for some $u\neq w \in \bfk\frakM(X)$. By the minimality of $m$, we have $w\neq v$. Using induction hypothesis, we get $\stars{q}{u} \astarrow_S \stars{q}{w}$ and $\stars{q}{w} \astarrow_S \stars{q}{v}$. This implies that $\stars{q}{u} \astarrow_S\stars{q}{v}$, as required.


(\mref{it:qsf}) Since $u\downarrow_S v$, we may suppose by Definition~\mref{def:ARSbasics}~(\mref{it:joinfg})
that $u \astarrow_S w$ and $v \astarrow_S w$
for some $w\in \bfk \frakM(X)$. Then by Item~(\mref{it:qs}), we have $\stars{q}{u} \astarrow_S \stars{q}{w}$ and  $\stars{q}{v} \astarrow_S \stars{q}{w}$. So $\stars{q}{u} \downarrow_S \stars{q}{v}$. This completes the proof.
\end{proof}

The following is a concept finer than subwords, including the information of placements~\mcite{ZheG}.

\begin{defn}
Let $w\in \mapm{X}$ such that
\begin{equation}
\stars{q_1}{u_1}=w=\stars{q_2}{u_2} \ \text{ for some }u_1, u_2\in \mapm{X}, q_1, q_2\in \frakM^\star(X).
\label{eq:plas}
\end{equation}
\noindent
The two \plas $(u_1,q_1)$ and $(u_2,q_2)$ are called
\begin{enumerate}
\item
\emph{ separated} if there exist $p \in \frakM^{\star_1,\star_2}(X)$ and $a,b \in \mapm{X}$ such that $\stars{q_1}{\star_1}=p|_{\star_1,\,b}$, $\stars{q_2}{\star_2}=p|_{a,\,\star_2}$, and $w=p|_{a,\,b}$;
\mlabel{it:bsep}

\item \emph{ nested} if there exists $q \in \frakM^{\star}(X)$ such that either $q_2=\stars{q_1}q$ or $q_1=\stars{q_2}q$;
\mlabel{it:bnes}

\item  \emph{ intersecting} if there exist $q \in \frakM^{\star}(X)$ and  $a, b, c \in \frakM(X)\backslash\{1\}$ such that $w=q|_{abc}$ and either
\begin{enumerate}
\item $ q_1=q|_{\star c}$ and $q_2=q|_{a\star}$; or

\item $q_1=q|_{a\star}$ and $q_2=q|_{\star c}$.
\end{enumerate}
\mlabel{it:bint}
\end{enumerate}
\mlabel{defn:bwrel}
\end{defn}

\begin{lemma}~{\rm \cite[Thm.~4.11]{ZheG}}
Let $w\in\frakM(X)$. For any two \plas $(u_1, q_1)$ and $(u_2,q_2)$ in $w$, exactly one of the following is true$\,:$
\begin{enumerate}
\item
$(u_1, q_1)$ and $(u_2,q_2)$ are separated$\,;$
\item
$(u_1, q_1)$ and $(u_2,q_2)$ are nested$\,;$
\item
$(u_1, q_1)$ and $(u_2,q_2)$ are intersecting. \mlabel{it:inter}
\end{enumerate}
\mlabel{lem:thrrel}
\end{lemma}

Now we fix some notations which will be used through out the remainder of the paper.
For any $u\in \frakM(X)$, define recursively $\lc u\rc^{(1)}:= \lc u\rc$ and
$\lc u\rc^{(k+1)}:= \lc \lc u\rc^{(k)}\rc$ for $k\geq 1$.
Recall from Example~\mref{ex:avera} that
\begin{equation*}
\phi(x_1, x_2) := \lc x_1\rc \lc x_2\rc - \lc \lc x_1\rc x_2\rc \,\text{ and }\,
\psi(x_1, x_2) := \lc x_1\lc x_2\rc\rc - \lc \lc x_1\rc x_2\rc \\
\end{equation*}
are the OPIs defining the averaging operator.
Let $\leq$ be a well-order on $X$ such that $x_1< x_2$. Then $\leq$ can be extended to
the monomial order $\leq_\db$ on $\frakM(X)$~\mcite{GGSZ}, which will be used through out in the remainder
of the paper.
With respect to $\leq_\db$, we have
\begin{equation}
\begin{aligned}
\lbar{\phi(x_1, x_2) } =& \lc x_1\rc \lc x_2\rc,\quad
\re{\phi(x_1, x_2)} =  \lc \lc x_1\rc x_2\rc,\\
\lbar{\psi(x_1, x_2) }=&  \lc x_1\lc x_2\rc\rc,\quad
\re{\psi(x_1, x_2)}= \lc \lc x_1\rc x_2\rc.
\end{aligned}
\mlabel{eq:phsi}
\end{equation}
The \term-rewriting system associated to  $\phi(x_1, x_2), \psi(x_1, x_2) $ is not confluent.
For example, for the element $\lc \lc x_1\rc\lc x_2\rc \rc\in \frakM(X)$, on the one hand,
$$\lc \lc x_1\rc\lc x_2\rc\rc \to_{\phi(x_1,x_2)} \lc \lc \lc x_1\rc x_2\rc\rc
=\lc \lc x_1\rc x_2\rc^{\two},$$
which is in normal form. On the other hand,
$$\lc \lc x_1\rc\lc x_2\rc\rc \to_{\psi(x_1,x_2)} \lc \lc\lc x_1\rc\rc x_2\rc
= \lc\lc x_1\rc^{\two} x_2\rc,$$
which is in normal form. So the element $\lc \lc x_1\rc\lc x_2\rc\rc$ is not confluent.
For the desired confluence, we need more rewriting rules. Let
\begin{equation}
\varphi(x_1, x_2) := \lc \lc x_1\rc x_2\rc^{\two} - \lc\lc x_1\rc^{\two} x_2\rc\,\text{ and }\, \Phi:=\{ \phi(x_1, x_2),\, \psi(x_1, x_2),\,\varphi(x_1, x_2)\}.
\mlabel{eq:vPhi}
\end{equation}
With respect to $\leq_\db$, we have
\begin{equation}
\lbar{\varphi(x_1, x_2)} = \lc \lc x_1\rc x_2\rc^{\two}  \,\text{ and }\,
\re{\varphi(x_1, x_2)} = \lc\lc x_1\rc^{\two} x_2\rc.
\mlabel{eq:varph}
\end{equation}
Let $u_1, u_2\in \frakM(X)$. Then by Eq.~(\mref{eq:genphi}),
$$\phi(u_1, u_2) = \lc u_1\rc\lc u_2\rc - \lc \lc u_1\rc u_2\rc\in S_\phi(X),$$
and by Lemma~\mref{lem:opideal},
$$\lc \lc u_1\rc\lc u_2\rc\rc - \lc \lc u_1\rc u_2\rc^{\two}
=\stars{\lc \star\rc\,}{\phi(u_1, u_2)}
\in \Id(S_\phi(X))\subseteq \Id(S_{\phi}(X)\cup S_{\psi}(X)).$$
With the same argument,
$$\lc \lc u_1\rc\lc u_2\rc\rc - \lc \lc u_1\rc^{\two}u_2\rc
= \psi(\lc u_1\rc, u_2) \in S_\psi(X)\subseteq \Id(S_\psi(X))
\subseteq \Id(S_{\phi}(X)\cup S_{\psi}(X)).$$
This implies that
\begin{align*}
& \varphi(u_1, u_2)= \lc \lc u_1\rc u_2\rc^{\two} - \lc \lc u_1\rc^{\two}u_2\rc \\
=& \lc \lc u_1\rc\lc u_2\rc\rc - \lc \lc u_1\rc^{\two}u_2\rc -
(\lc \lc u_1\rc\lc u_2\rc\rc - \lc \lc u_1\rc u_2\rc^{\two}) \in
\Id(S_{\phi}(X)\cup S_{\psi}(X))
\end{align*}
and so $\Id(S_\varphi(X)) \subseteq \Id(S_{\phi}(X)\cup S_{\psi}(X))$. Hence by Eqs.~(\mref{eq:genPhi}) and~(\mref{eq:vPhi}),
\begin{equation}
\Id(S_\Phi(X)) = \Id(S_{\phi}(X)\cup S_{\psi}(X)).
\mlabel{eq:idealeq}
\end{equation}

\begin{remark}
If $u_2 = 1$, then $\varphi(u_1, u_2)$ degenerates to
\begin{equation*}
\varphi(u_1, u_2) = \lc \lc u_1\rc u_2\rc^{\two} - \lc \lc u_1\rc^{\two}u_2\rc =
\lc u_1\rc^{\three} - \lc u_1\rc^{\three} = 0.
\end{equation*}
So we always assume $u_2\neq 1$ in $\varphi(u_1, u_2)$. This is our running hypothesis in
the remainder of the paper.
\mlabel{re:degen}
\end{remark}

\begin{remark} From Eqs.~(\mref{eq:phsi}) and~(\mref{eq:varph}), we have
\begin{enumerate}
\item for any $\firstx\in \Phi$ and $u_1,u_2\in \frakM(X)$, $\re{\firstu} \in \frakM(X)$ is a monomial. \mlabel{it:monom}

\item for any $u_1,u_2\in \frakM(X)$, the breadth $|\lbar{\phi(u_1,u_2)}| = 2$ and $|\lbar{\psi(u_1,u_2)}|= |\lbar{\varphi(u_1,u_2)}|= 1$. \mlabel{it:bre12}
\end{enumerate}
\mlabel{re:lebr}
\end{remark}

Recall $\Phi$ is fixed in Eq.~(\mref{eq:vPhi}).
In Eq.~(\mref{eq:rwS}), taking $S = S_\Phi(X)$ defined in Eq.~(\mref{eq:genPhi}),
 we get a \term-rewriting system associated to $\Phi$ (with respect to $\leq_\db$)
\begin{equation}
\begin{aligned}
\Pi_\Phi:= \Pi_{S_\Phi(X)}=\{\,q|_{\bar{\first(u_1,u_2)}}\to q|_{\re{\first(u_1,u_2)}} \mid \alpha(x_1, x_2)\in \Phi, q \in \frakM^\star(X), u_1,u_2\in\frakM(X) \}.
\end{aligned}
\mlabel{eq:T1}
\end{equation}
For notation clarity, we abbreviate $\to_{\firstu}$ as $\to_\first$.
Now we are in the position to consider the confluence of the \term-rewriting system $\Pi_\Phi$. By Theorem~\mref{thm:linec}, we only need to consider the confluence of basis
elements. Take a local fork of a basis element $w\in\frakM(X)$:
$$(\stars{q_1}{\re{\firstu}} \prescript{}{\first\,}\tto  \stars{q_1}{\lbar{\firstu}}= w= \stars{q_2}{\lbar{\secondv}}\to_\second \stars{q_2}{\re{\secondv}}),$$
where
$$\firstx, \secondx\in \Phi,\, u_i,v_i\in\frakM(X), \, i=1,2.$$
According to Lemma~\mref{lem:thrrel}, the two placements
$(\lbar{\firstu}, q_1)$ and $(\lbar{\secondv}, q_2)$ are separated, or intersecting, or nested.
We consider firstly the former two cases.

\begin{lemma}
Let $\first(x_1, x_2), \second(x_1, x_2)\in \Phi$ and $\stars{q_1}{{\lbar{\firstu}}} = \stars{q_2}{\lbar{\secondv}}$ for some $q_1, q_2\in \frakM^\star(X)$ and
$u_i,v_i\in \frakM(X)$, $i=1,2$. If the placements $({\lbar{\firstu}}, q_1)$ and $(\lbar{\secondv}, q_2)$ are separated, then $\stars{q_1}{\re{\firstu}} \downarrow_{\Phi} \stars{q_2}{\re{\secondv}}$.
\mlabel{lem:sepe}
\end{lemma}

\begin{proof}
In view of Definition~\ref{defn:bwrel}~(\mref{it:bsep}), there exists $p\in \frakM^{\star_1,\star_2}(X)$ such that
$$\stars{q_1}{\star_1}= \stars{p}{\star_1,\,\lbar{\secondv}}\,\text{ and }\,
\stars{q_2}{\star_2}= \stars{p}{\lbar{\firstu},\,\star_2}.
$$
On the one hand,
\begin{equation}
\begin{aligned}
\stars{q_1}{\re{\firstu}} = \stars{p}{\re{\firstu},\, \lbar{\secondv}} \to_{\second} \stars{p}{\re{\firstu},\, \re{\secondv}},
\end{aligned}
\mlabel{eq:seq1}
\end{equation}
where the last step employs the facts that $\re{\firstu}$ is a monomial by Remark~\mref{re:lebr}~(\mref{it:monom})
and so is $\stars{p}{\re{\firstu},\, \lbar{\secondv}}$.
On the other hand,
\begin{equation}
\begin{aligned}
\stars{q_2}{\re{\secondv}} = \stars{p}{\lbar{\firstu},\, \re{\secondv}} \to_{\first} \stars{p}{\re{\firstu},\, \re{\secondv}}.
\end{aligned}
\mlabel{eq:seq2}
\end{equation}
Comparing Eqs~(\mref{eq:seq1}) and~(\mref{eq:seq2}), we conclude that $\stars{q_1}{\re{\firstu}} \downarrow_{\Phi} \stars{q_2}{\re{\secondv}}$.
\end{proof}

\begin{lemma}
Let $\firstx, \secondx\in \Phi$ and $\stars{q_1}{\lbar{\first(u_1, u_2)}}= \stars{q_2}{\lbar{\second(v_1, v_2)}}$ for some $q_1, q_2\in \mstar$ and $u_i, v_i\in \frakM(X)$, $i=1,2$. If the placements $(\lbar{\first(u_1, u_2)}, q_1)$ and $(\lbar{\second(v_1, v_2)}, q_2)$ are intersecting, then $\stars{q_1}{\re{\firstu}} \downarrow_{\Phi} \stars{q_2}{\re{\secondv}}$.
\mlabel{lem:noint}
\end{lemma}

\begin{proof}
If the two placements $(\lbar{\firstu}, q_1)$ and $(\lbar{\secondv}, q_2)$ are intersecting,
by symmetry, we may assume that Item~$(\mref{it:bint})~({\rm i})$ in Definition~\mref{defn:bwrel} holds. Then $q_1 \neq q_2$, because if $q_1 = q_2$, then $\star c = a \star$, a contradiction.
So
\begin{equation}
\stars{q}{\lbar{\firstu}\, c } = \stars{q_1}{\lbar{\firstu}} = \stars{q_2}{\lbar{\secondv}} = \stars{q}{a\, \lbar{\secondv}} = \stars{q}{abc}
\mlabel{eq:qabcfs}
\end{equation}
and
$$ \lbar{\firstu}\, c  = a\, \lbar{\secondv} = abc.$$
This implies that
\begin{equation}
\lbar{\firstu} = ab\,\text{ and }\, \lbar{\secondv}= bc.
\mlabel{eq:intabc}
\end{equation}
If the breadth $|\lbar{\firstu}| =1$, then $a =1$ or $b=1$, both contradicting that $a,b\neq 1$ in Definition~\mref{defn:bwrel}~(\mref{it:bint}). Similarly, if the breadth $|\lbar{\secondv}| =1$, then $b =1$ or $c=1$, again a contradiction. So $|\lbar{\firstu}| \neq 1$ and
$|\lbar{\secondv}| \neq 1$. Hence by Remark~\mref{re:lebr}~(\mref{it:bre12}),
$$\firstx = \secondx= \phi(x_1, x_2) =
\lc x_1\rc \lc x_2\rc - \lc \lc x_1\rc x_2\rc.$$
From Eq.~(\mref{eq:intabc}), we have
$$ \lbar{\firstu} = \lc u_1\rc \lc u_2\rc = ab\,\text{ and }\, \lbar{\secondv} = \lc v_1\rc \lc v_2\rc = bc$$
and so $\lc u_1\rc = a$, $\lc u_2\rc = b = \lc v_1\rc$, $u_2 = v_1$ and
$  \lc v_2\rc  =c$. From Eqs.~(\mref{eq:phsi}) and~(\mref{eq:varph}),
$$\re{\firstu} c = \re{\phi(u_1, u_2)} c =
\lc \lc u_1 \rc u_2\rc \lc v_2\rc \tphi
\lc\lc \lc u_1 \rc u_2\rc v_2\rc $$
and
$$ a \re{\secondv} = a\re{\phi(v_1, v_2)}=a\re{\phi(u_2, v_2)}
 = \lc u_1\rc \lc \lc u_2\rc v_2\rc \tphi
\lc \lc u_1\rc \lc u_2\rc v_2\rc  \tphi
\lc\lc \lc u_1\rc u_2\rc v_2\rc.
 $$
So $\re{\firstu} c \downarrow_\Phi a \re{\secondv}$. This follows from Eq.~(\mref{eq:qabcfs}) and Lemma~\mref{lem:qs}~(\mref{it:qsf}) that
$$\stars{q_1}{\re{\firstu}} = \stars{q}{\re{\firstu}\, c} \downarrow_\Phi \stars{q}{a\, \re{\secondv}} =  \stars{q_2}{\re{\secondv}},$$
as required.
\end{proof}

Next, let us turn to consider the nested case. We need the following lemmas.
The first is on the leading monomials of OPIs in $\Phi$.

\begin{lemma}
Let $\first(x_1, x_2), \second(x_1, x_2)\in \Phi$ and $\lbar{\first(u_1, u_2)} = \lbar{\second(v_1, v_2)}$ for some $u_i, v_i\in \frakM(X)$, $i=1,2$. Then exactly one of the
following is true:
\begin{enumerate}
\item $\first(x_1, x_2)= \second(x_1, x_2)$, $u_1 = v_1$, $u_2=v_2$; \mlabel{it:ls1}

\item $\firstx = \psi(x_1, x_2)$, $\secondx = \varphi(x_1, x_2)$,
$u_1 = 1$, $u_2 = \lc v_1\rc v_2$; \mlabel{it:ls2}

\item $\firstx = \varphi(x_1, x_2)$, $\secondx = \psi(x_1, x_2)$,
$v_1=1$, $ v_2 = \lc u_1\rc u_2$. \mlabel{it:ls3}
\end{enumerate}
\mlabel{lem:noteq}
\end{lemma}

\begin{proof}
According to whether $\first$ and $\second$ are equal, we have the following cases to consider.

\smallskip\noindent
{\bf Case 1.} $\first(x_1, x_2) = \second(x_1, x_2)$. Then
Items~(\mref{it:ls2}) and (\mref{it:ls3}) fail.
We show Item~(\mref{it:ls1}) is valid.
Consider firstly that $\firstx = \phi(x_1, x_2)$. Then
$$\lc u_1\rc \lc u_2\rc =\lbar{\firstu}=\lbar{\secondv} =
\lc v_1\rc \lc v_2\rc.$$
By the unique decomposition of bracketed words in Eq.~(\mref{eq:udecom}), we have $\lc u_1\rc = \lc v_1\rc$ and
$\lc u_2\rc = \lc v_2\rc$. This implies $u_1 = v_1$ and $u_2 = v_2$.
Consider secondly that $\firstx = \psi(x_1, x_2)$. Then
$\lc  u_1\lc u_2\rc \rc =\lbar{\firstu}=\lbar{\secondv} =\lc v_1\lc v_2\rc\rc $
and so $  u_1\lc u_2\rc = v_1\lc v_2\rc$. This also implies $u_1 = v_1$, $\lc u_2\rc = \lc v_2\rc$ and $u_2 = v_2$. At last, consider $\firstx = \varphi(x_1, x_2)$.
Then $\lc \lc u_1\rc u_2\rc^{\two} =\lbar{\firstu}=\lbar{\secondv}=
\lc \lc v_1\rc v_2\rc^{\two}$ and so $\lc \lc u_1\rc u_2\rc = \lc \lc v_1\rc v_2\rc$.
Thus $\lc u_1\rc u_2 = \lc v_1\rc v_2$ and so $u_1 = v_1$ and $u_2 = v_2$.

\smallskip\noindent
{\bf Case 2.} $\first(x_1, x_2) \neq \second(x_1, x_2)$. Then Item~(\mref{it:ls1})
fails.

Suppose firstly that one of $\first(x_1, x_2)$ and $\second(x_1, x_2)$ is $\phi(x_1, x_2)$. By symmetry, we may let $\first(x_1, x_2) = \phi(x_1, x_2)$. Then $\second(x_1, x_2) \neq \phi(x_1, x_2)$. From Remark~\mref{re:lebr}~(\mref{it:bre12}), $|\lbar{\first(u_1, u_2)}| =|\lbar{\phi(u_1, u_2)}| = 2$ and $|\lbar{\second(v_1, v_2)}|=1$. This implies that $\lbar{\first(u_1, u_2)} \neq \lbar{\second(v_1, v_2)}$, contradicting our hypothesis.
Suppose
$\first(x_1, x_2),\second(x_1, x_2)\neq \phi(x_1, x_2)$. Then we have the following two subcases.

\smallskip\noindent
{\bf Case 2.1.}
 $\first(x_1, x_2) = \psi(x_1, x_2)$ and $\second(x_1, x_2) = \varphi(x_1, x_2)$.
Then Item~(\mref{it:ls3}) fails and
$$ \lc u_1\lc u_2\rc \rc = \lbar{\psi(u_1, u_2)}  =
\lbar{\first(u_1, u_2)}=\lbar{\secondv}= \lbar{\varphi(v_1, v_2)} =
\lc\lc v_1\rc v_2\rc^{\two}.$$
So $ u_1\lc u_2\rc =\lc\lc v_1\rc v_2\rc$. This implies that
$u_1 = 1$, $\lc u_2\rc = \lc \lc v_1\rc v_2\rc$ and $u_2 = \lc v_1\rc v_2$ and so Item~(\mref{it:ls2}) is valid.

\smallskip\noindent
{\bf Case 2.2.} $\first(x_1, x_2) = \varphi(x_1, x_2)$ and $\second(x_1, x_2) = \psi(x_1, x_2)$.
Then Item~(\mref{it:ls2}) fails and
$$ \lc \lc u_1\rc u_2\rc^{\two}= \lbar{\varphi(u_1, u_2)}
= \lbar{\first(u_1, u_2)}=\lbar{\secondv}= \lbar{\psi(v_1, v_2)}
= \lc v_1\lc v_2\rc\rc.$$
This follows that $ \lc \lc u_1\rc u_2\rc =  v_1\lc v_2\rc$.
So $v_1=1$, $ v_2 = \lc u_1\rc u_2$ and Item~(\mref{it:ls3}) is valid.
\end{proof}

%


\begin{lemma}
Let $\first(x_1, x_2), \second(x_1, x_2)\in \Phi$ and $\stars{q_1}{{\lbar{\firstu}}} = \stars{q_2}{\lbar{\secondv}}$ for some $q_1, q_2\in \frakM^\star(X)$ and
$u_i,v_i\in \frakM(X)$, $i=1,2$. If $q_2 = \stars{q_1}{q}$ for some $q\in\mstar$ and $\lbar{\secondv}$ is a subword of $u_1$ or $u_2$, then $\stars{q_1}{\re{\firstu}} \downarrow_{\Phi} \stars{q_2}{\re{\secondv}}$.
\mlabel{lem:incl0}
\end{lemma}

\begin{proof}
For clarity, write
$$\first:= \firstu\,\text{ and }\, \second:= \secondv.$$
By symmetry we may assume that $\bsecond$ is a subword of $u_1$ and so $u_1 = \stars{q'}{\bsecond}$ for some
$q'\in \mstar$. As $\firstx$ is linear on each variable and $\re{\firstu}$ is a
monomial by Remark~\mref{re:lebr}~(\mref{it:monom}), we may write
\begin{equation}
\bfirst =\lbar{\first(u_1, u_2)} = \stars{p}{u_1, u_2}\,\text{ and }\, \re{\first}=\re{\firstu} = \stars{p'}{u_1, u_2}\,\text{ for some }\, p,p'\in \mstar.
\mlabel{eq:first}
\end{equation}
Since $q_2 = \stars{q_1}{q}$ by our hypothesis, we have
$$\stars{q_1}{{\bfirst}} = \stars{q_2}{\bsecond} = \stars{q_1}{\stars{q}{\bsecond}},$$
and so
$$\stars{q}{\bsecond} = \bfirst =  \stars{p}{u_1, u_2} = \stars{p}{\stars{q'}{\bsecond}, u_2} =
 \stars{(\stars{p}{q', u_2})}{\bsecond}.$$
Hence
\begin{equation}
q = \stars{p}{q', u_2} = \lbar{\first(q', u_2)},
\mlabel{eq:qfirst}
\end{equation}
where the second equation employs Eq.~(\mref{eq:first}). So on the one hand, we have
\begin{equation}
\begin{aligned}
\stars{q_1}{\re{\first}} = \stars{q_1}{\stars{p'}{u_1, u_2}} = \stars{q_1}{\stars{p'}{\stars{q'}{\bsecond},\, u_2}} \to_{\second} \stars{q_1}{\stars{p'}{\stars{q'}{\re{\second}},\, u_2}},
\end{aligned}
\mlabel{eq:qrfir}
\end{equation}
where the first step is followed from Eq.~(\mref{eq:first}). On the other hand, we have
\begin{equation}
\begin{aligned}
\stars{q_2}{\re{\second}} = \stars{q_1}{\stars{q}{\re{\second}}} = \stars{q_1}{\lbar{\first(\stars{q'}{\re{\second}},\, u_2)}} \to_{\first}
\stars{q_1}{\re{\first(\stars{q'}{\re{\second}},\, u_2)}}
=\stars{q_1}{\stars{p'}{\stars{q'}{\re{\second}},\, u_2}},
\end{aligned}
\mlabel{eq:qrsec}
\end{equation}
where the first step is followed from the hypothesis $q_2 = \stars{q_1}{q}$, the second from Eq.~(\mref{eq:qfirst}) and the last from Eq.~(\mref{eq:first}). Comparing Eqs~(\mref{eq:qrfir}) and~(\mref{eq:qrsec}), we obtain $\stars{q_1}{\re{\first}} \downarrow_{\Phi} \stars{q_2}{\re{\second}}$.  This completes the proof.

\end{proof}

As an application of Theorem~\mref{thm:linec}, we have

\begin{theorem}
The \term-rewriting system $\Pi_\Phi$ defined in Eq.~(\mref{eq:T1}) is convergent.
\mlabel{thm:pconf}
\end{theorem}

\begin{proof}
Since $\leq_\db$ we used is a monomial order on $\frakM(X)$, $\Pi_\Phi$ is terminating~\mcite{GGSZ}.
By Definition~\mref{defn:ARS}~(\mref{it:dconv}), we are left to show that $\Pi_\Phi$ is confluent.
From Theorem~\mref{thm:linec}, it is sufficient to prove that $\Pi_\Phi$ is locally confluent for any basis element. Let
$$(\stars{q_1}{\re{\firstu}} \prescript{}{\Phi\,}\tto  \stars{q_1}{\lbar{\firstu}}= w= \stars{q_2}{\lbar{\secondv}}\to_\Phi \stars{q_2}{\re{\secondv}})$$
be an arbitrary local fork of a basis element $w$, where
$$\firstx, \secondx\in \Phi,\, u_i,v_i\in\frakM(X), \, i=1,2.$$
We only need to show that
\begin{equation}
\stars{q_1}{\re{\firstu}} \downarrow_\Phi \stars{q_2}{\re{\secondv}}.
\mlabel{eq:wantc0}
\end{equation}
According to Lemma~\mref{lem:thrrel}, the two placements
$(\lbar{\firstu}, q_1)$ and $(\lbar{\secondv}, q_2)$ are separated, or nested, or intersecting.
If they are separated or intersecting, then by Lemmas~\mref{lem:sepe} and~\mref{lem:noint},
Eq.~(\mref{eq:wantc0}) holds.
If the two placements $(\lbar{\firstu}, q_1)$ and $(\lbar{\secondv}, q_2)$ are nested,
by symmetry in Definition~\mref{defn:bwrel}~(\mref{it:bnes}), we may assume that $q_2 = \stars{q_1}{q}$. If $\lbar{\secondv}$ is a subword of $u_1$ or $u_2$, then by Lemma~\mref{lem:incl0}, Eq.~(\mref{eq:wantc0}) holds.

Suppose $\lbar{\secondv}$ is not a subword of $u_1$ and $u_2$.
Note
\begin{equation}
\stars{q_1}{\lbar{\firstu}} = \stars{q_2}{\lbar{\secondv}} =  \stars{q_1}{\stars{q}{\lbar{\secondv}}}\,\text{ and so}\, \lbar{\firstu} =\stars{q}{\lbar{\secondv}}.
\mlabel{eq:abq}
\end{equation}
Since $q_2 = \stars{q_1}{q}$, Eq.~(\mref{eq:wantc0}) is equivalent to
$$\stars{q_1}{\re{\firstu}} \downarrow_\Phi \stars{q_1}{\stars{q}{\re{\secondv}}}.$$
So to prove Eq.~(\mref{eq:wantc0}), by Lemma~\mref{lem:qs}~(\mref{it:qsf}), it is enough to show that
\begin{equation}
\re{\firstu} \downarrow_\Phi \stars{q}{\re{\secondv}}.
\mlabel{eq:wantc}
\end{equation}

If $q=\star$, then $\lbar{\firstu} = \lbar{\secondv}$.
By Lemma~\mref{lem:noteq}, exactly one of the three items there holds.
If Item~(\mref{it:ls1}) holds, then $\re{\first(u_1, u_2)} = \re{\second(v_1, v_2)} $ and Eq.~(\mref{eq:wantc}) is valid by $q = \star$. Since Item~(\mref{it:ls2}) and Item~(\mref{it:ls3}) are symmetric,
we consider that Item~(\mref{it:ls2}) holds. Then
$$\firstx = \psi(x_1, x_2),\, \secondx = \varphi(x_1, x_2),\,
u_1 = 1, u_2 = \lc v_1\rc v_2.$$
This follows from Eqs.~(\mref{eq:phsi}) and~(\mref{eq:varph}) that
$$\re{\firstu} = \lc \lc u_1\rc u_2\rc = \lc \lc 1\rc \lc v_1\rc v_2\rc
\tphi \lc \lc\lc 1\rc v_1\rc v_2\rc
$$
and
$$\stars{q}{\re{\secondv}} =
\stars{\star\,}{\lc \lc v_1\rc^{\two} v_2\rc} =
\lc \lc v_1\rc^{\two} v_2\rc = \lc \lc 1 \lc v_1\rc\rc v_2\rc
\tpsi \lc \lc \lc 1 \rc v_1\rc v_2\rc.$$
Hence Eq.~(\mref{eq:wantc}) is valid.

Summing up, we are left to consider the case of that
\begin{equation}
q_2 = \stars{q_1}{q},\,\lbar{\firstu} = \stars{q}{\lbar{\secondv}},\, q\neq \star \,\text{ and }\, \lbar{\secondv} \text{ is not a subword of } u_1 \text{ and } u_2.
\mlabel{eq:divid}
\end{equation}
Then
\begin{equation}
q_1\neq q_2 \,\text{ and }\, \lbar{\firstu}\neq \lbar{\secondv}.
\mlabel{eq:fsneq}
\end{equation}
We have the following cases to consider.

\smallskip\noindent
{\bf Case 1.} $\firstx = \phi(x_1, x_2)$. Then $\lbar{\firstu} = \lc u_1\rc\lc u_2\rc $ by Eq.~(\mref{eq:phsi}).

If $\secondx = \phi(x_1, x_2)$, then
$$\lc u_1\rc\lc u_2\rc = \lbar{\firstu} = \stars{q}{\lbar{\secondv}} =
\stars{q}{\lc v_1\rc\lc v_2\rc},$$
that is, $\lc v_1\rc\lc v_2\rc$ is a subword of $\lc u_1\rc\lc u_2\rc $. By Eq.~(\mref{eq:fsneq}), $\lc v_1\rc\lc v_2\rc \neq \lc u_1\rc\lc u_2\rc$. So $\lc v_1\rc\lc v_2\rc$ is a subword of $\lc u_1\rc$ or $\lc u_2\rc$. Since $\lc v_1\rc\lc v_2\rc\neq \lc u_1\rc, \lc u_2\rc$ by comparing the breadth, $\lc v_1\rc\lc v_2\rc$ is a subword of $u_1$ or $u_2$ by Lemma~\mref{lem:qubra}~(\mref{it:qubra}), contradicting Eq.~(\mref{eq:divid}). So $\secondx \neq \phi(x_1, x_2)$.

\smallskip\noindent
{\bf Subcase 1.1.} $\secondx = \psi(x_1, x_2)$.
In this subcase, we have
\begin{equation}
\lc u_1\rc\lc u_2\rc =\lbar{\firstu} =\stars{q}{\lbar{\secondv}}
= \stars{q}{\lc v_1 \lc v_2\rc\rc},
\mlabel{eq:qeq1.2}
\end{equation}
that is, $\lc v_1 \lc v_2\rc\rc$ is a subword of $\lc u_1\rc\lc u_2\rc$.
By Lemma~\mref{lem:qubra}~(\mref{it:youbr}), either $\lc v_1 \lc v_2\rc\rc$ is a
subword of $\lc u_1\rc$ or
$\lc v_1 \lc v_2\rc\rc$ is a subword of $\lc u_2\rc$. Note that $\lbar{\secondv} = \lc v_1 \lc v_2\rc\rc$ is not a subword of $u_1$ and $u_2$ by Eq.~(\mref{eq:divid}). From Lemma~\mref{lem:qubra}~(\mref{it:qubra}) and Eq.~(\mref{eq:qeq1.2}),
either
\begin{equation}
\lc v_1 \lc v_2\rc\rc = \lc u_1\rc\,\text{ and }\, q = \star\lc u_2\rc,
\mlabel{eq:case11}
\end{equation}
or
\begin{equation}
\lc v_1 \lc v_2\rc\rc = \lc u_2\rc\,\text{ and }\,q =\lc u_1\rc\star.
\mlabel{eq:case12}
\end{equation}
For the former case of Eq.~(\mref{eq:case11}), we have
$$\re{\phi(u_1,u_2)} = \lc \lc u_1\rc u_2\rc =  \lc \lc v_1 \lc v_2\rc\rc u_2\rc \tpsi
\lc \lc \lc v_1 \rc v_2\rc u_2\rc$$
and
$$ \stars{q}{\re{\psi(v_1,v_2)}} =\stars{(\star\lc u_2\rc)\,}{\lc \lc v_1\rc  v_2\rc }
= \lc \lc v_1\rc  v_2\rc \lc u_2\rc \tphi \lc \lc\lc v_1\rc  v_2\rc u_2\rc.$$
Hence $\re{\phi(u_1,u_2)}  \downarrow_\Phi \stars{q}{\re{\psi(v_1,v_2)}}$ and Eq.~(\mref{eq:wantc}) holds, as needed.
For the later case of Eq.~(\mref{eq:case12}), we have $u_2 = v_1 \lc v_2\rc$. So
$$\re{\phi(u_1, u_2)} = \lc \lc u_1\rc u_2\rc =  \lc \lc u_1\rc v_1 \lc v_2\rc\rc
\tpsi \lc \lc \lc u_1\rc v_1 \rc v_2\rc $$
and
$$\stars{q}{\re{\psi(v_1, v_2)}}= \stars{(\lc u_1\rc \star)\,}{\lc \lc v_1\rc v_2\rc}
= \lc u_1\rc \lc \lc v_1\rc v_2\rc \tphi \lc \lc u_1\rc \lc v_1\rc v_2\rc
\tphi  \lc \lc\lc u_1\rc v_1\rc v_2\rc.$$
Hence $\re{\phi(u_1, u_2)} \downarrow_\Phi \stars{q}{\re{\psi(v_1, v_2)}}$ and Eq.~(\mref{eq:wantc}) holds, as needed.

\smallskip\noindent
{\bf Subcase 1.2.} $\secondx = \varphi(x_1, x_2)$.
In this subcase, we have
\begin{equation}
\lc u_1\rc\lc u_2\rc = \lbar{\firstu} =\stars{q}{\lbar{\secondv}}
= \stars{q}{\lc \lc v_1\rc v_2\rc^{\two}},
\mlabel{eq:qeq1.3}
\end{equation}
that is, $\lc \lc v_1\rc v_2\rc^{\two}$ is a subword
of $\lc u_1\rc\lc u_2\rc$. By Lemma~\mref{lem:qubra}~(\mref{it:youbr}), either
$\lc \lc v_1\rc v_2\rc^{\two}$ is a subword of $\lc u_1\rc$ or
$\lc \lc v_1\rc v_2\rc^{\two}$ is a subword of or $\lc u_2\rc$.
Since $\lc \lc v_1\rc v_2\rc^{\two}$ is not a subword or $u_1$ and $u_2$ by Eq.~(\mref{eq:divid}), from Lemma~\mref{lem:qubra}~(\mref{it:qubra}) and Eq~(\mref{eq:qeq1.3}), either
\begin{equation}
\lc \lc v_1\rc v_2\rc^{\two} = \lc u_1\rc \,\text{ and }\, q = \star \lc u_2\rc
\mlabel{eq:pvc1}
\end{equation}
or
\begin{equation}
\lc \lc v_1\rc v_2\rc^{\two} = \lc u_2\rc \,\text{ and }\, q = \lc u_1\rc \star.
\mlabel{eq:pvc2}
\end{equation}
Consider firstly the former case of Eq.~(\mref{eq:pvc1}). We have
\begin{align*}
\re{\phi(u_1, u_2)} = \lc \lc u_1\rc u_2\rc = \lc \lc \lc v_1\rc v_2\rc^{\two} u_2\rc
\tvarphi  \lc \lc \lc v_1\rc^{\two} v_2\rc u_2\rc
\end{align*}
and
\begin{align*}
\stars{q}{\re{\varphi(v_1, v_2)}} =& \stars{(\star\lc u_2\rc)\, }{\lc \lc v_1 \rc^{\two} v_2\rc} = \lc \lc v_1 \rc^{\two} v_2\rc\lc u_2\rc \tphi
\lc\lc \lc v_1 \rc^{\two} v_2\rc u_2\rc.
\end{align*}
Hence $\re{\phi(u_1, u_2)} \downarrow_\Phi \stars{q}{\re{\varphi(v_1, v_2)}}$ and Eq.~(\mref{eq:wantc}) holds.
For the later case of Eq.~(\mref{eq:pvc2}), we have $u_2 = \lc \lc v_1\rc v_2\rc$. Then
\begin{align*}
\re{\phi(u_1, u_2)} =&\, \lc\lc  u_1\rc u_2\rc = \lc\lc  u_1\rc \lc \lc v_1\rc v_2\rc\rc
\tphi \lc\lc \lc u_1\rc \lc v_1\rc v_2\rc\rc =  \lc \lc u_1\rc \lc v_1\rc v_2\rc^{\two}
\tphi \lc \lc \lc u_1\rc v_1\rc v_2\rc^{\two} \\
\tvarphi& \, \lc \lc \lc u_1\rc v_1\rc^{\two} v_2\rc
\tvarphi \lc \lc \lc u_1\rc^{\two} v_1\rc v_2\rc
\end{align*}
and
\begin{align*}
\stars{q}{\re{\varphi(v_1, v_2)}} = &\,
\stars{(\lc u_1\rc \star)\,}{\lc \lc v_1\rc^{\two} v_2\rc}= \lc u_1\rc \lc \lc v_1\rc^{\two} v_2\rc \tphi
\lc\lc u_1\rc \lc v_1\rc^{\two} v_2\rc
\tphi \lc\lc\lc u_1\rc\lc v_1\rc\rc v_2\rc \\
\tphi&\, \lc\lc\lc\lc u_1\rc v_1\rc\rc v_2\rc =  \lc\lc\lc u_1\rc v_1\rc^{\two} v_2\rc
\tvarphi \lc\lc\lc u_1\rc^{\two} v_1\rc v_2\rc.
\end{align*}
Hence $\re{\phi(u_1, u_2)}  \downarrow_\Phi \stars{q}{\re{\varphi(v_1, v_2)}}$ and Eq.~(\mref{eq:wantc}) holds, as needed.

\smallskip\noindent
{\bf Case 2.} $\firstx = \psi(x_1, x_2)$.
Then $\lbar{\firstu} = \lc u_1\lc u_2\rc\rc $ by Eq.~(\mref{eq:phsi}).

\smallskip\noindent
{\bf Case 2.1.} $\secondx = \phi(x_1, x_2)$.
In this subcase, we have
\begin{equation}
\lc u_1\lc u_2\rc\rc =\lbar{\firstu} =\stars{q}{\lbar{\secondv}}= \stars{q}{\lc v_1\rc\lc v_2\rc},
\mlabel{eq:qeq2.1}
\end{equation}
that is, $\lc v_1\rc\lc v_2\rc$ is a subword of $\lc u_1\lc u_2\rc\rc$. Since $\lc v_1\rc\lc v_2\rc\neq \lc u_1\lc u_2\rc\rc$ by Eq.~(\mref{eq:fsneq}), it follows from Lemma~\mref{lem:qubra}~(\mref{it:qubra}) that $\lc v_1\rc\lc v_2\rc$ is a subword of
$ u_1\lc u_2\rc$. Note $\lc v_1\rc\lc v_2\rc$ is not a subword of $u_1$ or $u_2$
by Eq.~(\mref{eq:divid}).
So $a\lc v_1\rc\lc v_2\rc =  u_1\lc u_2\rc$ for some $a\in \frakM(X)$ and $q =\lc a\star \rc $ by Eq.~(\mref{eq:qeq2.1}).
Then $a\lc v_1\rc = u_1$, $\lc v_2\rc = \lc u_2\rc$, $v_2 = u_2$.
This follows that
$$\re{\psi(u_1, u_2)} = \lc \lc u_1\rc u_2\rc = \lc \lc a\lc v_1\rc \rc u_2\rc
\tpsi \lc \lc \lc a\rc v_1\rc u_2\rc$$
and
$$\stars{q}{\re{\phi(v_1, v_2)}} = \stars{(\lc a \star\rc)\,}{\lc \lc v_1\rc v_2\rc}
= \lc a \lc \lc v_1\rc v_2\rc\rc = \lc a \lc \lc v_1\rc u_2\rc\rc \tpsi \lc\lc a \rc \lc v_1\rc u_2\rc
\tphi \lc\lc \lc a \rc v_1\rc u_2\rc .$$
Hence $\re{\psi(u_1, u_2)} \downarrow_\Phi \stars{q}{\re{\phi(v_1, v_2)}}$ and Eq.~(\mref{eq:wantc}) holds, as needed.

\smallskip\noindent
{\bf Case 2.2} $\secondx= \psi(x_1, x_2)$.
In this subcase, we have
\begin{equation}
\lc u_1\lc u_2\rc\rc =\lbar{\firstu} =\stars{q}{\lbar{\secondv}}=
\stars{q}{\lc v_1\lc v_2\rc\rc},
\mlabel{eq:qeq2.2}
\end{equation}
that is, $\lc v_1\lc v_2\rc\rc$ is a subword of
$\lc u_1\lc u_2\rc\rc $.
By Lemma~\mref{lem:qubra}~(\mref{it:qubra}) and
$\lc v_1\lc v_2\rc\rc \neq \lc u_1\lc u_2\rc\rc$ from Eq.~(\mref{eq:fsneq}),
$\lc v_1\lc v_2\rc\rc$ is a subword of $u_1\lc u_2\rc$.
Note $\lc v_1\lc v_2\rc\rc$ is not a subword of $u_1$ and $u_2$ by Eq.~(\mref{eq:divid}).
So by Lemma~\mref{lem:qubra}~(\mref{it:youbr}),
$\lc v_1\lc v_2\rc\rc$ is a subword of $\lc u_2\rc$.
From Lemma~\mref{lem:qubra}~(\mref{it:qubra}),
we have $\lc v_1\lc v_2\rc\rc = \lc u_2\rc$, $ v_1\lc v_2\rc = u_2$ and
$q = \lc u_1 \star \rc $ by Eq.~(\mref{eq:qeq2.2}). Thus
$$ \re{\psi(u_1, u_2)} = \lc \lc u_1\rc u_2\rc = \lc \lc u_1\rc v_1\lc v_2\rc\rc
\tpsi \lc \lc\lc u_1\rc v_1 \rc v_2\rc$$
and
$$ \stars{q}{\re{\psi(v_1, v_2)}} =
\stars{(\lc u_1 \star \rc)\,}{\lc \lc v_1\rc v_2\rc}
= \lc u_1 \lc \lc v_1\rc v_2\rc \rc \tpsi \lc \lc u_1 \rc \lc v_1\rc v_2\rc
\tphi  \lc \lc\lc u_1 \rc v_1\rc v_2\rc .$$
Hence $\re{\psi(u_1, u_2)} \downarrow_\Phi \stars{q}{\re{\psi(v_1, v_2)}}$ and Eq.~(\mref{eq:wantc}) holds, as needed.

\smallskip\noindent
{\bf Case 2.3.} $\secondx = \varphi(x_1,x_2)$.
In this subcase, we have
\begin{equation}
\lc  u_1\lc u_2\rc\rc  =\lbar{\firstu} =\stars{q}{\lbar{\secondv}}=
\stars{q}{\lc \lc v_1\rc  v_2\rc^{\two}},
\mlabel{eq:qeq2.3}
\end{equation}
that is, $\lc \lc v_1\rc  v_2\rc^{\two}$ is a subword of
$\lc  u_1\lc u_2\rc\rc $.
By Lemma~\mref{lem:qubra}~(\mref{it:qubra}) and
$\lc \lc v_1\rc  v_2\rc^{\two} \neq \lc  u_1\lc u_2\rc\rc$ from Eq.~(\mref{eq:fsneq}),
$\lc \lc v_1\rc  v_2\rc^{\two}$ is a subword of
$ u_1\lc u_2\rc$.
Note from Eq.~(\mref{eq:divid}), $\lc \lc v_1\rc  v_2\rc^{\two}$ is not
a subword of $u_1$ and $u_2$.
So by Lemma~\mref{lem:qubra}~(\mref{it:youbr}),
$\lc \lc v_1\rc  v_2\rc^{\two}$ is a subword of $ \lc u_2\rc$.
 By Lemma~\mref{lem:qubra}~(\mref{it:qubra}),
$  \lc \lc v_1\rc  v_2\rc^{\two} = \lc u_2\rc$ and then $q = \lc u_1 \star\rc$ by Eq.~(\mref{eq:qeq2.3}). This implies
$ \lc \lc v_1\rc  v_2\rc = u_2 $. Thus
\begin{align*}
 \re{\psi(u_1,u_2)} = &\, \lc \lc u_1\rc u_2\rc =
 \lc \lc u_1\rc \lc \lc v_1\rc  v_2\rc \rc \tphi
  \lc\lc \lc u_1\rc \lc v_1\rc  v_2\rc \rc = \lc \lc u_1\rc \lc v_1\rc  v_2\rc^{\two}\\
 \tphi&\, \lc\lc \lc u_1\rc v_1\rc  v_2\rc^{\two}
 \tvarphi \lc\lc \lc u_1\rc v_1\rc^{\two}  v_2\rc
  \tvarphi \lc\lc \lc u_1\rc^{\two} v_1\rc  v_2\rc
\end{align*}
and
\begin{align*}
\stars{q}{\re{\varphi(v_1,v_2)}} =& \,
\stars{(\lc u_1\star\rc)\,}{\lc \lc v_1 \rc^{\two} v_2\rc }
= \lc u_1\lc \lc v_1 \rc^{\two} v_2\rc\rc \tpsi
\lc \lc u_1\rc \lc v_1 \rc^{\two} v_2\rc
\tphi \lc \lc\lc u_1\rc\lc v_1 \rc\rc v_2\rc \\
\tphi&\, \lc\lc \lc\lc u_1\rc v_1 \rc\rc v_2\rc
= \lc \lc\lc u_1\rc v_1 \rc^{\two} v_2\rc
\tvarphi \lc \lc\lc u_1\rc^{\two} v_1 \rc v_2\rc.
\end{align*}
Hence $ \re{\psi(u_1,u_2)} \downarrow_\Phi \stars{q}{\re{\varphi(v_1,v_2)}} $ and Eq.~(\mref{eq:wantc}) holds, as needed.

\smallskip\noindent
{\bf Case 3.} $\firstx = \varphi(x_1,x_2)$. Then $\lbar{\firstu} = \lc \lc u_1\rc u_2\rc^{\two}$ by Eq.~(\mref{eq:varph}).

\smallskip\noindent
{\bf Subcase 3.1.} $\secondx = \phi(x_1,x_2)$.
In this subcase,
\begin{equation}
\lc \lc u_1\rc u_2\rc^{\two} =\lbar{\firstu} =\stars{q}{\lbar{\secondv}}=\stars{q}{\lc v_1\rc \lc v_2\rc },
\mlabel{eq:qeq3.1}
\end{equation}
that is, $\lc v_1\rc \lc v_2\rc $ is a subword of $\lc \lc u_1\rc u_2\rc^{\two}$.
As $\lc v_1\rc \lc v_2\rc\neq \lc \lc u_1\rc u_2\rc^{\two}$ by Eq.~(\mref{eq:divid}),
$\lc v_1\rc \lc v_2\rc$ is a subword of $\lc \lc u_1\rc u_2\rc$ by Lemma~\mref{lem:qubra}~(\mref{it:qubra}).
Again using Lemma~\mref{lem:qubra}~(\mref{it:qubra}),
$\lc v_1\rc \lc v_2\rc$ is a subword of $\lc u_1\rc u_2$ by $\lc v_1\rc \lc v_2\rc \neq \lc \lc u_1\rc u_2\rc$.
From Eq.~(\mref{eq:divid}), $\lbar{\secondv} = \lc v_1\rc \lc v_2\rc$ is not a subword of $u_1$ and $u_2$.
Hence $\lc v_1\rc \lc v_2\rc a= \lc u_1\rc u_2$ for some $a\in \frakM(X)$ and so $q = \lc \star a\rc^{\two}$ by Eq.~(\mref{eq:qeq3.1}). This implies that
$\lc v_1\rc = \lc u_1\rc$, $v_1 = u_1$ and $ \lc v_2\rc a = u_2.$
Hence
$$\re{\varphi(u_1, u_2)} = \lc \lc u_1\rc^{\two} u_2\rc=
\lc \lc u_1\rc^{\two} \lc v_2\rc a\rc \tphi\lc \lc\lc u_1\rc^{\two} v_2\rc a\rc
$$
and
\begin{align*}
\stars{q} {\re{\phi(v_1, v_2)}}=&\, \stars{(\lc\star a\rc^{\two})\,}{\lc \lc v_1\rc v_2\rc}
=\lc\lc \lc v_1\rc v_2\rc a \rc^{\two}
=\lc\lc \lc u_1\rc v_2\rc a \rc^{\two} \\
\tvarphi&\, \lc\lc \lc u_1\rc v_2\rc^{\two} a \rc
\tvarphi \lc\lc \lc u_1\rc^{\two} v_2\rc a \rc.
\end{align*}
Hence $\re{\varphi(u_1, u_2)} \downarrow_\Phi \stars{q} {\re{\phi(v_1, v_2)}}$ and Eq.~(\mref{eq:wantc}) holds, as needed.

\smallskip\noindent
{\bf Subcase 3.2.} $\secondx = \psi(x_1, x_2)$.
In this subcase,
\begin{equation}
\lc \lc u_1\rc u_2\rc^{\two}=\lbar{\firstu} =
\stars{q}{\lbar{\secondv}} = \stars{q}{\lc v_1\lc v_2\rc\rc},
\mlabel{eq:qeq3.2}
\end{equation}
that is, $\lc v_1\lc v_2\rc\rc$ is a subword of $\lc \lc u_1\rc u_2\rc^{\two}$.
Since $\lc v_1\lc v_2\rc\rc\neq \lc \lc u_1\rc u_2\rc^{\two}$ by Eq.~(\mref{eq:fsneq}),
$\lc v_1\lc v_2\rc\rc$ is a subword of $\lc \lc u_1\rc u_2\rc$ by Lemma~\mref{lem:qubra}~(\mref{it:qubra}).
Again using Lemma~\mref{lem:qubra}~(\mref{it:qubra}), either
$\lc v_1\lc v_2\rc\rc = \lc \lc u_1\rc u_2\rc$ or
$\lc v_1\lc v_2\rc\rc$ is a subword of
$\lc u_1\rc u_2$.

For the former case of $\lc v_1\lc v_2\rc\rc = \lc \lc u_1\rc u_2\rc$, we have $q = \lc \star\rc$ by Eq.~(\mref{eq:qeq3.2}) and $v_1\lc v_2\rc= \lc u_1\rc u_2$.
This implies that $v_1 = \lc u_1\rc v_1'$ and $u_2 = u_2' \lc v_2\rc$ for some $v_1',u_2'\in \frakM(X)$. Then
$$ \lc u_1\rc v_1'\lc v_2\rc= v_1\lc v_2\rc= \lc u_1\rc u_2 = \lc u_1\rc u_2' \lc v_2\rc\,
\text{ and so }\, v_1' = u_2'=:a.$$
Then $v_1 = \lc u_1\rc a$ and $u_2 = a\lc v_2\rc$. This follows that
$$ \re{\varphi(u_1,u_2)} = \lc\lc u_1\rc^{\two} u_2\rc =
\lc\lc u_1\rc^{\two}a\lc v_2\rc\rc \tpsi \lc\lc\lc u_1\rc^{\two}a\rc v_2\rc
$$
and
\begin{align*}
\stars{q}{\re{\psi(v_1,v_2)}} =&\, \stars{\lc \star\rc\,}{\lc\lc v_1\rc v_2\rc}
= \lc \lc\lc v_1\rc v_2\rc\rc
=\lc\lc v_1\rc v_2\rc^{\two} = \lc\lc \lc u_1\rc a\rc v_2\rc^{\two}\\
\tvarphi&\, \lc\lc \lc u_1\rc a\rc^{\two} v_2\rc
\tvarphi  \lc\lc \lc u_1\rc^{\two} a\rc v_2\rc.
\end{align*}
Hence $\re{\varphi(u_1,u_2)} \downarrow_\Phi \stars{q}{\re{\psi(v_1,v_2)}}$ and Eq.~(\mref{eq:wantc}) holds, as needed.

Consider the latter case that $\lc v_1\lc v_2\rc\rc$ is a subword of
$\lc u_1\rc u_2$. By Eq.~(\mref{eq:divid}), $\lbar{\secondv} = \lc \lc v_1\rc v_2\rc$ is not a
subword of $u_1$ and $u_2$. So from Lemma~\mref{lem:qubra}~(\mref{it:youbr}),
$\lc v_1\lc v_2\rc\rc$ is a subword of $\lc u_1\rc$.
Using Lemma~\mref{lem:qubra}~(\mref{it:qubra}), we have
$\lc v_1\lc v_2\rc\rc = \lc u_1\rc$ and so $q = \lc \star u_2 \rc^{\two}$ by Eq.~(\mref{eq:qeq3.2}).
Then $ v_1\lc v_2\rc = u_1$. So we have
$$ \re{\varphi(u_1,u_2)} = \lc \lc u_1\rc^{\two} u_2\rc
=  \lc \lc  v_1\lc v_2\rc\rc^{\two} u_2\rc \tpsi
\lc \lc  \lc v_1\rc v_2\rc^{\two} u_2\rc \tvarphi
\lc \lc  \lc v_1\rc^{\two} v_2\rc u_2\rc
$$
and
$$\stars{q}{\re{\psi(v_1,v_2)}} =
 \stars{( \lc\star u_2\rc^{\two})\,}{\lc \lc v_1\rc v_2\rc}=
  \lc\lc \lc v_1\rc v_2\rc u_2\rc^{\two} \tvarphi
  \lc\lc \lc v_1\rc v_2\rc^{\two} u_2\rc
  \tvarphi   \lc\lc \lc v_1\rc^{\two} v_2\rc u_2\rc.
 $$
Hence $\re{\varphi(u_1,u_2)} \downarrow_\Phi \stars{q}{\re{\psi(v_1,v_2)}}$ and Eq.~(\mref{eq:wantc}) holds, as needed.

\smallskip\noindent
{\bf Subcase 3.3.} $\secondx = \varphi(x_1, x_2)$. In this subsection, we have
\begin{equation}
\lc \lc u_1\rc u_2\rc^{\two} = \lbar{\firstu} =
\stars{q}{\lbar{\secondv}} = \stars{q}{\lc\lc v_1\rc v_2\rc^{\two} },
\mlabel{eq:qeq3.3}
\end{equation}
that is, $\lc\lc v_1\rc v_2\rc^{\two} $ is a subword of $\lc \lc u_1\rc u_2\rc^{\two}$.
By Eq.~(\mref{eq:fsneq}), $\lc\lc v_1\rc v_2\rc^{\two} \neq \lc \lc u_1\rc u_2\rc^{\two}$.
So from Lemma~\mref{lem:qubra}~(\mref{it:qubra}),
$\lc\lc v_1\rc v_2\rc^{\two} $ is a subword of $\lc \lc u_1\rc u_2\rc$.
Again using Lemma~\mref{lem:qubra}~(\mref{it:qubra}), either
$\lc\lc v_1\rc v_2\rc^{\two}  = \lc \lc u_1\rc u_2\rc$ or
$\lc\lc v_1\rc v_2\rc^{\two} $ is a subword of $\lc u_1\rc u_2$.

For the former case, we have $q = \lc \star\rc$ by Eq.~(\mref{eq:qeq3.3})
and $\lc\lc v_1\rc v_2\rc =  \lc u_1\rc u_2$. This implies that $u_2 = 1$,
$\lc\lc v_1\rc v_2\rc = \lc u_1\rc$ and $\lc v_1\rc v_2 = u_1$. Then
$$\re{\varphi(u_1, u_2)} = \lc \lc u_1\rc^{\two} u_2\rc = \lc u_1\rc^{\three}
 = \lc\lc v_1\rc v_2\rc^{\three} \tvarphi \lc\lc v_1\rc^{\two} v_2\rc^{\two} \tvarphi
 \lc\lc v_1\rc^{\three} v_2\rc
 $$
and
$$\stars{q}{\re{\varphi(v_1, v_2)}} = \stars{\lc\star\rc\,}{\lc \lc v_1\rc^{\two} v_2\rc}
= \lc\lc \lc v_1\rc^{\two} v_2\rc\rc = \lc \lc v_1\rc^{\two} v_2\rc^{\two}
\tvarphi  \lc \lc v_1\rc^{\three} v_2\rc.$$
Hence $\re{\varphi(u_1,u_2)} \downarrow_\Phi \stars{q}{\re{\varphi(v_1,v_2)}}$ and Eq.~(\mref{eq:wantc}) holds, as needed.

Consider the later case of that
$\lc\lc v_1\rc v_2\rc^{\two} $ is a subword of $\lc u_1\rc u_2$.
By Eq.~(\mref{eq:divid}), $\lbar{\secondv} = \lc\lc v_1\rc v_2\rc^{\two}$
is not a subword of $u_1$ and $u_2$. So from Lemma~\mref{lem:qubra}~(\mref{it:youbr}),
$\lc\lc v_1\rc v_2\rc^{\two}$ is a subword of $\lc u_1\rc$.
Using Lemma~\mref{lem:qubra}~(\mref{it:qubra}), we have
$\lc\lc v_1\rc v_2\rc^{\two} = \lc u_1\rc$ and so
$q= \lc \star u_2\rc^{\two}$ by Eq.~(\mref{eq:qeq3.3}).
Thus we have
$$\re{\varphi(u_1, u_2)} = \lc \lc u_1\rc^{\two} u_2\rc =
\lc \lc\lc v_1\rc v_2\rc^{\three} u_2\rc \tvarphi
\lc \lc\lc v_1\rc^{\two} v_2\rc^{\two} u_2\rc \tvarphi
\lc \lc\lc v_1\rc^{\three} v_2\rc u_2\rc
$$and
$$\stars{q}{\re{\varphi(v_1, v_2)}}
= \stars{( \lc \star u_2\rc^{\two})\,}{\lc \lc v_1\rc^{\two} v_2\rc}
=  \lc \lc \lc v_1\rc^{\two} v_2\rc u_2\rc^{\two} \tvarphi
\lc \lc \lc v_1\rc^{\two} v_2\rc^{\two} u_2\rc \tvarphi
\lc \lc \lc v_1\rc^{\three} v_2\rc u_2\rc
$$
Hence $\re{\varphi(u_1, u_2)} \downarrow_\Phi \stars{q}{\re{\varphi(v_1, v_2)}}$ and Eq.~(\mref{eq:wantc}) holds, as needed. This completes the proof.

\end{proof}

Recall from Remark~\mref{re:degen} that $u_2\neq 1$ in $\varphi(u_1, u_2)$. So we
define
\begin{equation}
\begin{aligned}
M:=& \{\stars{q}{\lbar{\phi(u_1, u_2)}},\,\stars{q}{\lbar{\psi(u_1, u_2)}} \mid
q\in \frakM^\star(X), u_1, u_2\in \frakM(X)\},\\
N:=& \{\stars{q}{\lbar{\varphi(u_1, u_2)}} \mid
q\in \frakM^\star(X), u_1\in \frakM(X), u_2\in \frakM(X)\setminus\{1\}\},\\
N_1:=& \{\stars{q}{\lbar{\varphi(u_1, u_2)}} \mid
q\in \frakM^\star(X), u_1,u_2\in \frakM(X)\},\\
N_2:=& \{\stars{q}{\lbar{\varphi(u_1, 1)}} \mid
q\in \frakM^\star(X), u_1\in \frakM(X)\}.
\end{aligned}
\end{equation}
Then $N = N_1\setminus N_2$. From Eqs.~(\mref{eq:phsi}) and~(\mref{eq:varph}),
$$\stars{q}{\lbar{\varphi(u_1, 1)}} = \stars{q}{\lc u_1\rc^{\three}}
= \stars{q}{\lc 1\lc u_1\rc^{\two} \rc }
= \stars{q}{\lbar{\psi(1,\,\lc u_1\rc) }}\in M$$
and so $N_2 \subseteq M$. Thus
\begin{equation*}
M\cup N = M\cup ( N_1\setminus N_2)= M\cup N_1.
%
\end{equation*}
Hence
\begin{equation}
\begin{aligned}
 &\{\stars{q}{\lbar{s}} \mid q\in \mstar, s\in S_\Phi(X) \}\\
= & \{\stars{q}{\lbar{s}} \mid q\in \mstar, s\in S_\phi(X) \cup S_\psi(X) \}
\cup \{\stars{q}{\lbar{s}} \mid q\in \mstar, s\in S_\varphi(X) \}\\
=& M\cup N = M\cup N_1\\
=& \{\stars{q}{\lc u_1\rc \lc u_2\rc },\,
\stars{q}{\lc u_1\lc u_2\rc \rc},\, \stars{q}{\lc \lc u_1 \rc u_2 \rc^{\two}}
\mid q\in \mstar, u_1, u_2\in \frakM(X) \},
\end{aligned}
\mlabel{eq:mneq}
\end{equation}
where the second step employs Remark~\mref{re:degen}.
Now we are ready to give our main result. From Proposition~\mref{pp:freea} and Eq.~(\mref{eq:idealeq}), $\bfk \frakM(X)/ \Id(S_\Phi(X))$
is the free averaging algebra on $X$.

\begin{theorem}
The $\Irr(S_\Phi(X))$ is a \bfk-basis of the free (unitary) averaging algebra $\bfk \frakM(X)/ \Id(S_\Phi(X))$ on $X$. More precisely,
$$\bfk \frakM(X) =  \Id(S_\Phi(X)) \oplus \bfk \Irr(S_\Phi(X)), $$
where
$$\Irr(S_\Phi(X)) = \frakM(X) \setminus
\{\stars{q}{\lc u_1\rc \lc u_2\rc },\,
\stars{q}{\lc u_1\lc u_2\rc \rc},\, \stars{q}{\lc \lc u_1 \rc u_2 \rc^{\two}}
\mid q\in \mstar, u_1, u_2\in \frakM(X) \}.$$
\mlabel{thm:basis}
\end{theorem}

\begin{proof}
By Theorem~\mref{thm:pconf}, $\Pi_{\Phi} = \Pi_{S_\Phi(X)}$ is convergent. Using Theorems~\mref{thm:convsum} to $S = S_\Phi(X)$, we have
$$\bfk \frakM(X) =  \Id(S_\Phi(X)) \oplus \bfk \Irr(S_\Phi(X)),$$
where
\begin{align*}
\Irr(S_\Phi(X)) = & \frakM(X) \setminus \{\stars{q}{\lbar{s}} \mid q\in \mstar, s\in S_\Phi(X) \}\\
=& \frakM(X) \setminus
\{\stars{q}{\lc u_1\rc \lc u_2\rc },\,
\stars{q}{\lc u_1\lc u_2\rc \rc},\, \stars{q}{\lc \lc u_1 \rc u_2 \rc^{\two}}
\mid q\in \mstar, u_1, u_2\in \frakM(X) \}
\end{align*}
by Eq.~(\mref{eq:mneq}).
\end{proof}


\noindent {\bf Acknowledgements}: This work was supported by the National Natural Science Foundation of China (Grant No. 11201201, 11371177 and 11371178) and
the Natural Science Foundation of Gansu Province (Grant No. 1308RJZA112).

%

\end{document}